\documentclass [12pt,a4paper,reqno]{amsart}
 \input amssymb.sty
\input xy
\xyoption{all}

\def\mm{\frak m}

\def\endbox{ \hfill\quad\qed}

\def\an{{\operatorname{an}}}
\def\san{{\operatorname{super-an}}}
\def\vth{\vartheta}
\def\sig{\sigma}

\def\Orb{\operatorname{Orb}}

\def\vsim{\sim_v}
\def\bR{\overline R}
\def\Rmq{{R \setminus \gq}}
\def\bRmq{\overline{R \setminus \gq}}

\def\({\left(}
\def\){\right)}

\def\MFC{\operatorname{MFC}}
\def\OFC{\operatorname{OFC}}

\def\Uv{U(v)}
\def\bUv{\overline {U(v)}}
\def\bvrp{\overline \vrp}
\def\tvrp{\widetilde \vrp}
\def\bU{\overline U}
\def\bx{\bar x}
\def\by{\bar y}
\def\be{\bar e}
\def\bv{\bar v}

\def\m{$m$}
\def\N{\mathbb N}
\def\R{\mathbb R}
\def\bt{\beta}

\def\UTG{GS}
\def\UTGR{ghost surpassing relation}

\def\gq{\mathfrak q}

\def\go{\mathfrak o}

\def\tE{E_\tng}

\def\GS{GS}

\def\MFC{\operatorname{MFC}}
\def\MFCE{MFCE}
\def\STR{\operatorname{STR}}
\def\trc{\operatorname{trc}}
\def\lc{\operatorname{lc}}
\def\corn{\operatorname{Corn}}

\def\tng{{\operatorname{t}}}
\def\stg{{\operatorname{s}}}

\def\tE{E_\tng}

\def\tCov{\Cov_\tng}
\def\sCov{\Cov_\stg}
\def\tsCov{\Cov_{\tng,\stg}}
\def\stCov{\Cov_{\tng \stg}}

\def\ep{\varepsilon}

\def\al{\alpha}
\def\gm{\gamma}
\def\Gm{\Gamma}
\def\vrp{\varphi}
\def\tv{\tilde{v}}

\def\pipeGS{{\underset{\operatorname{\, gs }}{\mid}}}

\def\lmodg{\mathrel  \pipeGS \joinrel \joinrel =}

\def\Real{\mathbb R}

\newcommand{\etype}[1]{\renewcommand{\labelenumi}{(#1{enumi})}}
\def\eroman{\etype{\roman}}
\def\ealph{\etype{\alph}}

\def\tT{\mathcal T}

\def\tG{\mathcal G}

\def\pipeGS{{\underset{\operatorname{\, gs }}{\mid}}}

\def\lmod{\mathrel  \mid    \joinrel =}
\def\lmodg{\mathrel  \pipeGS   \joinrel \joinrel =}

\newtheorem{thm}{Theorem} [section]
\newtheorem*{thm*}{Theorem}
\newtheorem{cor}[thm]{Corollary}
\newtheorem{lem}[thm]{Lemma}
\newtheorem{lemma}[thm]{Lemma}
\newtheorem{prop}[thm]{Proposition}

\newtheorem*{claim*} {Claim}
\newtheorem*{theorem13.5'} {Theorem 13.5$'$}
\newtheorem{acknowledgment*}[thm] {Acknowledgment}

\newtheorem{example}[thm]{Example}
\newtheorem{examp}[thm]{Example}
\newtheorem{subexamp}[thm]{Subexample}
\newtheorem{examples}[thm]{Examples}
 \newtheorem{remark}[thm]{Remark}
  \newtheorem{remarks}[thm]{Remarks}
 \newtheorem*{remark*}{Remark}
 \newtheorem{defn}[thm]{Definition}
 \newtheorem*{defn*}{Definition}

\newtheorem{construction}[thm]{Construction}
\newtheorem{terminology}[thm]{Terminology}
\newtheorem{schol}[thm]{Scholium}
\newtheorem{notation}[thm]{Notation}

\newtheorem*{notation*} {Notation}
\newtheorem*{comm*} {Comment}

\newtheorem{rem}[thm]{Remark}

\newcommand{\thmref}[1]{Theorem~\ref{#1}}
\newcommand{\propref}[1]{Proposition~\ref{#1}}

\newcommand{\lemref}[1]{Lemma~\ref{#1}}

 \renewcommand{\sectionmark}[1]{}

\newcommand{\bfem}[1]{\textbf{#1}}

\newcommand{\iy}{\infty}

\newcommand{\diag}{\operatorname{diag}}

\newcommand{\Cov}{\operatorname{Cov}}
\newcommand{\Card}{\operatorname{Card}}

\newcommand{\lm}{\lambda}

 \newcommand{\id}{\operatorname{id}}

 \newcommand{\supp} {\operatorname{supp}}

\newcommand{\osr}{\overset\sim \rightarrow}

\renewcommand{\a}{\alpha}

\begin{document}

\title[Supertropical semirings and supervaluations]
{Supertropical semirings and supervaluations}

\author[Z. Izhakian]{Zur Izhakian}
\address{Department of Mathematics, Bar-Ilan University, 52900 Ramat-Gan,
Israel} \email{zzur@math.biu.ac.il}
\author[M. Knebusch]{Manfred Knebusch}
\address{Department of Mathematics,
NWF-I Mathematik, Universit\"at Regensburg, 93040 Regensburg,
Germany} \email{manfred.knebusch@mathematik.uni-regensburg.de}
\author[L. Rowen]{Louis Rowen}
 \address{Department of Mathematics,
 Bar-Ilan University, 52900 Ramat-Gan, Israel}
 \email{rowen@macs.biu.ac.il}

\thanks{This research of the first and third authors is supported  by the
Israel Science Foundation (grant No.  448/09).}
\thanks{This research of the second author was supported in part by
 the Gelbart Institute at
Bar-Ilan University, the Minerva Foundation at Tel-Aviv
University, the Department of Mathematics   of Bar-Ilan
University, and the Emmy Noether Institute at Bar-Ilan
University.}
\thanks{This paper was completed under the auspices of the Resarch
in Pairs program of the Mathematisches Forschungsinstitut
Oberwolfach, Germany.}

\subjclass[2010]{Primary: 13A18, 13F30, 16W60, 16Y60; Secondary:
03G10, 06B23, 12K10,   14T05}
\date{\today}


\keywords{Supertropical algebra, monoids, bipotent semirings,
valuation theory,  supervaluations, transmissions, lattices}


\begin{abstract}
We interpret a valuation $v$ on a ring $R$ as a map $v: R \to M$
into a so called  bipotent semiring $M$ (the usual max-plus
setting), and then define a \textbf{supervaluation}  $\vrp$ as a
suitable map into a supertropical semiring $U$ with ghost ideal
$M$ (cf. \cite{IzhakianRowen2007SuperTropical},
\cite{IzhakianRowen2008Matrices}) covering $v$ via the ghost map
$U \to M$. The set  $\Cov(v)$ of all supervaluations covering $v$
has a natural ordering which makes it a  complete lattice. In the
case that $R$ is a field, hence for $v$ a Krull valuation, we give
a completely explicit description of $\Cov(v)$.


The theory of supertropical semirings and supervaluations  aims
for an algebra fitting the needs of tropical geometry better than
the usual max-plus setting. We illustrate this by giving a
supertropical version of Kapranov's Lemma.

\end{abstract}

\maketitle

{\small \tableofcontents}

\numberwithin{equation}{section}

\section*{Introduction}

As explained in \cite{IMS} and \cite{Gathmann:0601322}, tropical
geometry grew out of a logarithmic correspondence taking a
polynomial $f(\lm _1, \dots, \lm _n)$ over the ring of Puiseux
series to a corresponding polynomial $\bar f(\lm _1, \dots, \lm
_n)$ over the max-plus algebra $T$. A key observation is
Kapranov's Lemma, that this correspondence sends the algebraic
variety defined by $f$ into the so-called {\it corner locus}
defined by $\bar f$. More precisely, this correspondence involves
  the negative of a  valuation (where the target ($T$) is an
ordered monoid rather than an ordered group), which has led
researchers in tropical mathematics to utilize valuation theory.
In order to avoid the introduction of the negative, some
researchers, such as \cite{SS}, have used the min-plus algebra
instead of the max-plus algebra. There is a deeper result which
describes the image of this correspondence; several versions
appear in the literature, one of which is in~\cite{Payne08}.

Note that whereas a valuation $v$ satisfies $v(ab) = v(a) + v(b),$
one only has $$v(a+b) = \min \{ v(a), v(b)\}$$ when $v(a) \ne
v(b);$ for the case that $v(a) = v(b)$, $v(a+b)$ could be any
element $\ge v(a).$ From this point of view, the max-plus (or,
dually, min-plus) algebra does not precisely reflect the tropical
mathematics.
 In order to deal with this issue, as well as to enhance the algebraic structure of the
max-plus algebra $T$, the first author introduced a cover of $T$,
graded by the multiplicative monoid $(\mathbb Z_2,\cdot),$ which
was dubbed the {\it extended tropical arithmetic}. Then, in
\cite{IzhakianRowen2007SuperTropical} and
\cite{IzhakianRowen2008Matrices}, this structure has been
amplified to the notion of \textbf{supertropical semiring}. A
supertropical semiring $U$ is equipped with a ``\textbf{ghost
map}"  $\nu := \nu_U : U \to U$, which respects addition and
multiplication and is idempotent, i.e., $\nu \circ \nu = \nu$.
Moreover $a + a = \nu(a)$ for every  $a \in U$ (cf. \S3). This
rule replaces the rule $a+ a= a$ in the usual max-plus (or
min-plus) arithmetic.  We call $\nu(a)$ the ``\textbf{ghost}"  of
$a$ (often writing $a^\nu$ instead of $\nu(a)$), and we call the
elements of $U$ which are not ghost
``\textbf{tangible}"\footnote{The element $0$ may be regarded both
as tangible and ghost.}.

The image of the ghost map is a so-called \textbf{bipotent
semiring}, i.e., a semiring $M$ such that $a + b \in \{ a, b\}$
for every $a, b \in M$. So $M$ is a semiring typically occurring
in tropical algebra. In this paper supertropical and bipotent
semirings are nearly  always tacitly  assumed to be commutative.

It turns out that supertropical semirings allow a refinement of
valuation theory to a theory of ``supervaluations".
Supervaluations  seem to be  able to give an enriched version of
tropical geometry. In the present paper we illustrate this by
giving a refined and generalized version of Kapranov's Lemma (\S9,
\S11). Very roughly, one may say that the usual tropical algebra
is present in the ghost level of our supertropical setting.

We consider valuations on rings (as defined by Bourbaki \cite{B})
instead of just fields. We mention that these are essential for
understanding families of valuations on fields,
cf. e.g. \cite{HK} and \cite{KZ}. We use multiplicative
notation, writing  a valuation $v$ on a ring  $R$  as a map into
$\Gm \cup \{ 0 \}$ with $\Gm$ a multiplicative ordered abelian
group and $0 < \Gm$, obeying the rules
  \begin{equation}\renewcommand{\theequation}{$*$}\addtocounter{equation}{-1}\label{eq:str.1}
  \begin{array}{c}
    v(0) = 0, \quad v(1) = 1, \quad v(ab) = v(a) v(b),  \\[2mm]
     v(a+b)
\leq \max(v(a),v(b)). \\
  \end{array}
\end{equation}
We view the ordered monoid $\Gm \cup \{ 0\}$ as a bipotent
semiring by introducing  the addition $x + y := \max(x,y)$, cf.
\S1 and \S2. It is then  very natural to replace $\Gm \cup  \{ 0
\}$ by any bipotent semiring $M$, and to define an
\textbf{\m-valuation} (= monoid valuation) $v: R \to M$ in the
same way $(*)$ as before.


Given an \m-valuation $v : R \to M$ there exist multiplicative
mappings $\vrp: R \to M$ into various supertropical semirings $U$,
with $\vrp(0) = 0$, $\vrp(1) =1$, such that  $M$ is the ghost
ideal of $U$ and $\nu_U \circ \vrp = v$. These are the
\textbf{supervaluations} covering $v$, cf. \S4.

In \S5 we define maps $\al:U \to V$ between supertropical
semirings, called  \textbf{transmissions}, which have the property
that for a supervaluation $\vrp:R \to U$ the composite $\al \circ
\vrp : R \to V$ is again a supervaluation. Given two
supervaluations $\vrp: R \to U$ and $\psi: R \to V$ (not
necessarily covering the same \m-valuation $v$), we say that
$\vrp$ \textbf{dominates} $\psi$, and write $\vrp \geq \psi$, if
there exists a transmission $\al: U \to V$, such that  $\psi = \al
\circ \vrp$. \{The transmission $\al$ then is essentially
unique.\}

Restricting the  dominance relation to the set of
supervaluations\footnote{More precisely we should consider
equivalence classes of supervaluations. We suppress this point
here.} covering a fixed \m-valuation $v: R \to M$ we obtain a
partially ordered set $\Cov(v)$, which  turns out to be a complete
lattice, as proved in \S7. The bottom element of this lattice is
the \m-valuation $v$, viewed as a supervaluation. The top element,
denoted $\vrp_v: R \to U(v)$, can be described explicitly in good
cases. This description is already given in \S4, cf. Example
\ref{examp4.5}. The other elements of $\Cov(v)$ are obtained from
$\vrp_v$ by dividing out suitable equivalence relations on the
semiring $U(v)$, called \MFCE-relations (= multiplicative fiber
conserving equivalence relations). They are defined in \S6.
Finally in \S8, we obtain an explicit description of all elements
of $\Cov(v)$ in the case that $R$ is a field, in which case $v$ is
a Krull valuation.

If $R$ is only a ring, our results are far less complete.
Nevertheless it seems to be absolutely necessary to work at least
in this generality for many reasons, in particular functorial
ones, cf. e.g. \cite{HK}, \cite{KZ}.

In \S9 we delve deeper into the supertropical theory to pinpoint a
relation, which we call the \textbf{\UTGR}, which seems to be a
key for working in supertropical semirings. On the one hand, the
{\UTGR}\ restricts to equality on tangible elements, thereby
enabling us to specialize to the max-plus theory. On the other
hand, the \UTGR\ appears in virtually every supertropical theorem
proved so far, especially in supertropical matrix theory in
\cite{IzhakianRowen2008Matrices} and
\cite{IzhakianRowen2009Equations}.

In the present paper the ghost surpassing relation is the
essential gadget for understanding and proving a general version
of Kapranov´s  Lemma in \S11 (Theorem 11.15, preceded by Theorem
9.11), valid for any valuation $v: R \to M$ which is
``\textbf{strong}". This means that $v(a+b) = \max(v(a), v(b))$
whenever $v(a) \neq v(b)$. If $R$ is a ring, every valuation on
$R$ is strong, as is very well known, but if $R$ is only a
semiring, this is a restrictive condition.   On our way to
Kapranov's Lemma we employ supervaluations $\vrp \in \Cov(v)$
which are \textbf{tangible}, i.e., have only tangible values, and
are \textbf{tangibly additive}, which means  that $\vrp(a+b) =
\vrp(a) + \vrp(b)$ whenever $\vrp(a) + \vrp(b)$ is tangible. We
apostrophize tangibly additive supervaluations which cover strong
\m-valuations as \textbf{strong supervaluations}.

The strong  tangible supervaluations in $\Cov(v)$ seem to be the
most suitable ones  for applications in tropical geometry also
beyond Kapranov's Lemma, as to be explained at the end of this
introduction. They form a sublattice $\tsCov(v)$ of $\Cov(v)$. In
particular there exists an ``initial" tangible strong valuation in
$\Cov(v)$, denoted by $\overline \vrp _v$, which dominates all
others. It gives the ``best" supertropical version of Kapranov's
Lemma, cf. \S9. At the end of \S10 we make $\overline \vrp _v$
explicit in the case that $v$
 is the natural valuation  of the field of formal Puiseux series  in a
 variable $t$ (with real or with rational exponents). We can
 interpret the value of $\overline \vrp _v (a(t))$ of a Puiseux
 series $a(t)$ as the leading term of $a(t)$, while $v(a(t))$ can
 be seen as the $t$-power contained in the leading term.

Strictly speaking, Kapranov's Lemma extends the valuation $v$ to
the polynomial ring $R[\lm _1, \dots, \lm _n]$ over $R$, with
target in the polynomial ring $M[\lm _1, \dots, \lm _n]$, which no
longer is bipotent. Thus, the theory in this paper needs to be
generalized if we are to deal formally with such notions. This is
set forth in the last Section \ref{sec:13}, in which the target of
a valuation is replaced by a monoid with a binary sup operation.

Since the theory of tropicalization has developed recently in
terms of the  valuations on the field of Puiseux series, let us
indicate briefly how this theory can be extended to the
supertropical environment.

We recall  the algebraic theory of analytification, as presented
by Payne \cite{Payne09}. A \textbf{multiplicative seminorm}
$|\phantom{e}|: R\to \Real_{\geq 0}$ on a ring $R$ is a
multiplicative map satisfying the triangle inequality $$|a+b| \le
|a| + |b|$$ for all $a,b\in R.$ In particular,  any \m-valuation
$v: R \to \Real_{\geq 0}$ is a seminorm. \{Recall that we use
multiplicative notation.\} If $X$ ia an affine variety over a
field $K$ (e.g. $K$ is a field of generalized Puiseux series over
an algebraically closed field) and $v: K \to  \Real_{\geq 0}$
 is a valuation, then Payne's space $K[X]^\an$ is the set of all
 multiplicative seminorms  on $K[X]$ that extend $v$. More
 generally, if $v:K \to M$ is a valuation  with $M= \tG \cup \{0
 \}$ any bipotent semifield (cf. \S1), then we may define a space
 $K[X]^\an$ associated to $(K,v)$ and $X$ in exactly the same way.

 But in the supertropical context we can do more. Let
 $$U:= D(\tG) = \STR(\tG, \tG, \id_\tG),$$
 as defined in Example \ref{examps3.16}. This is a supertropical
 semifield with ghost part $\tG \cup \{ 0 \} = M$. We define a
 space $K[X]^\san$ as the set of all strong supervaluations $\vrp: K[X] \to
 U$ such that the valuation $w: K[X] \to M$ covered by $\vrp$
 (cf. Definition \ref{defn4.1}) is an element of
 $K[X]^\an$, i.e., $w$ extends $v$.

We have the natural map $K[X]^\san \to K[X]^\an$, given by $\vrp
\mapsto w$, which exhibits $K[X]^\san$ as a ``covering'' of
Payne's space $K[X]^\an$. But there is still another relation
between these two spaces, which seems to be more intriguing. The
set $U$ contains a second copy of $M$ as a multiplicative
submonoid, namely the tangible part  $\tT(U) \cup \{ 0\} \cong \tG
\cup \{ 0 \}$. Interpreting the elements of $K[X]^\an$ as maps
from $K[X]$ to $\tT(U) \cup \{ 0\}$  we may view $K[X]^\an$ as the
set of all tangible supervaluations $\vrp: K[X] \to
 U$ (automatically strong) with $\vrp|K$ covering $v$. Thus
 $K[X]^\an$ can be seen as the subspace of $K[X]^\san$ consisting of
 the $\vrp \in K[X]^\san$ which do not have ghost values.

Of course, nothing can prevent us from replacing $K$ by any ring,
or even semiring $R$, and $v$ by any strong \m-valuation on $R$,
and defining $K[X]^\an$ and $K[X]^\san$ in this generality.

 The reader may ask  whether
valuations and supervaluations on semirings instead of just rings
deserve interest apart from formal issues. They do. It is only for
not making a long paper even longer that we do not give
applications to semirings here.

The semiring  $R = \sum  A^2$ of sum of squares of a commutative
ring (or even a field) $A$ with $-1 \notin R$ is a case  in point.
  Real algebra seems to be a
fertile ground for studying valuations and supervaluations on
semirings. The paper contains only one very small  hint  pointing
in this direction, Example~\ref{example2.2}.


\section{Bipotent semirings}\label{bilpotentsemirings}

Let $R$ be a semiring (always with unit element $1=1_R).$ Later we
will assume that $R$ is commutative, but presently this is not
necessary.

\begin{defn}\label{defn1.1} We call a  pair $(a,b)\in
R^2$ \bfem{bipotent} if $a+b\in\{a,b\}.$ We call the
 semiring  $R$ \bfem{bipotent} if every pair $(a,b)\in R^2$
is bipotent.
\end{defn}

\begin{prop}\label{prop1.2} Assume that $R$ is a bipotent
semiring. Then the binary relation $(a,b\in R)$
\begin{equation}\label{1.1}
a\le b\quad\text{iff}\quad a+b=b
\end{equation}
on $R$ is a total ordering on the set $R,$ compatible with
addition and multiplication, i.e., for all $a,b,c\in R$
\begin{align*} &a\le b\quad\Rightarrow\quad ac\le bc,\ ca\le cb,
\\ & a\le b\quad\Rightarrow \quad a+c\le b+c.
\end{align*}
\end{prop}

\begin{proof} A straightforward check.\end{proof}

\begin{remark}\label{rem1.3}
We can define such a binary relation $\le$ by \eqref{1.1} in any
semiring, and then obtain a partial ordering compatible with
addition and multiplication. The ordering is total iff $R$ is
bipotent. Clearly, $0_R\le x$ for every $ x\in  R.$
\end{remark}

\begin{defn}\label{defn1.4} We call a semiring $R$ a
\bfem{semidomain}, if $R$ has no zero divisors, i.e., the set
$R\setminus \{0\}$ is closed under multiplication. We call $R$ a
\bfem{semifield}, if $R$ is commutative and every element $x\ne0$
of $R$ is invertible; hence $R\setminus\{0\}$ is a group under
multiplication.
\end{defn}

Given a bipotent semidomain $R$, the set $G:=R\setminus\{0\}$ is a
totally ordered monoid under the multiplication of $R.$

In this way we obtain all (totally) ordered monoids. Indeed, if
$G=(G,\cdot)$ is a given ordered  monoid, we gain a bipotent
semiring $R$ as follows: Adjoin a new element $0$ to $G$ and form
the set $R:=G\cup\{0\}.$ Extend the multiplication on $G$ to a
multiplication on~$R$ by the rules $0\cdot g=g\cdot0=0$ for any
$g\in G$ and $0\cdot0=0.$ Extend the ordering of $G$ to a total
ordering on~$R$ by the rule $0<g$ for $g\in G.$ Then define an
addition on~$R$ by the rule
$$x+y:=\max(x,y)$$
for any $x,y\in R.$ It is easily checked that $R$ is a bipotent
semiring, and that the ordering on $R$ by the rule \eqref{1.1} is
the given one. We denote this semiring $R$ by $T(G).$

These considerations can be easily amplified to the following
theorem.

\begin{thm}\label{thm1.4}
The category of (totally) ordered monoids $G$ is
isomorphic\footnote{This is more than equivalent!} to the category
of bipotent semidomains $R$ by the assignments
$$G\mapsto T(G),\quad R\mapsto R\setminus\{0\}.$$

Here the morphisms in the first category by definition are the
order preserving monoid homomorphisms $\gamma: G\to G'$ in the
weak sense; i.e., $\gamma$ is multiplicative, $\gamma(1)=1,$ and
$x\le y\Rightarrow \gamma(x)\le \gamma(y),$ while the morphisms in
the second category are the semiring homomorphisms (with
$1\mapsto1).$
\end{thm}

In the following we regard an ordered monoid and the associated
bipotent semiring as the same entity in a different disguise.
Usually we prefer the semiring viewpoint.

\begin{example}\label{examp1.5} Starting with the monoid $G=(\mathbb
R,+)$, i.e., the field of real numbers with the usual addition, we
obtain a bipotent semifield
$$T(\mathbb R):=\mathbb R\cup\{-\infty\},$$
where addition $\oplus$ and  multiplication  $\odot$ of $T(\mathbb
R)$ are defined as follows, and the neutral element of addition is
denoted by $-\infty$ instead of $0,$ since our monoid is now given
in additive notation. For $x,y\in\mathbb R$
\begin{align*}
x\oplus y&=\max(x,y),\\
x\odot y&=x+y,\\
(-\infty)\oplus x&=x\oplus(-\infty)=x,\\
(-\infty)\odot x&=x \odot (-\infty)=-\infty,\\
(-\infty)\oplus(-\infty)&=-\infty,\\
(-\infty)\odot(-\infty)&=-\infty.
\end{align*}
\end{example}

$T(\mathbb R)$ is the ``real tropical semifield" of common
tropical algebra, often called the ``max-plus" algebra $\mathbb
R\cup\{-\infty\}:$ cf.~ \cite{IMS}, or \cite{SS}  (there a
``min-plus" algebra is used).

\section{\m-valuations}\label{sec:mval}

In this section we assume that all occurring semirings and monoids
are commutative.

 Let $R$  be a semiring.

\begin{defn}\label{defn2.1} An \bfem{\m-valuation} (= monoid
valuation) on $R$ is a map $v: R\to M$ into a (commutative)
bipotent semiring $M\ne\{0\}$ with the following properties:
\begin{align*}
&V1: v(0)=0,\\
&V2: v(1)=1,\\
&V3:v(xy)=v(x)v(y)\quad\forall x,y\in R,\\
&V4: v(x+y)\le v(x)+v(y)\quad [=\max(v(x),v(y))]\quad \forall
x,y\in R.
\end{align*}
We call the \m-valuation $v$ \bfem{strict}, if instead of V4 the
following stronger axiom holds:
\begin{align*}
 V5: v(x+y)=v(x)+v(y)\quad \forall x,y\in R. & \qquad \qquad
 \qquad \qquad
\end{align*}
\end{defn}

Note that a strict \m-valuation $v: R\to M$ is just a semiring
homomorphism from $R$ to~$M.$
\medskip

In the special case that $M=\Gamma\cup\{0\}$ with $\Gamma$ an
ordered abelian group, we call the \linebreak \m-valuation $v:
R\to M$ a \bfem{valuation}. Notice that in the case that $R$ is a
ring (instead of a semiring), this is exactly the notion of a
valuation as defined by Bourbaki \cite{B} (Alg. Comm. VI, \S3,
No.1) and studied, e.g., in \cite{HK} and \cite[Chap.~I]{KZ},
except that for $\Gamma$ we have chosen the multiplicative
notation instead of the additive notation.

If $v: R\to M$ is an \m-valuation, we may replace $M$ by the
submonoid $v(R).$ We then speak of $v$ as a \bfem{surjective}
\bfem{\m-valuation}.

\begin{defn}\label{defn2.2}
A (commutative) monoid $G$ is called \bfem{cancellative}, if, for
any $a,b,c\in G$, the equation $ac=bc$ implies $a=b.$
\end{defn}

Notice that an ordered monoid $G$ is cancellative iff $a<b$
implies $ac<bc$ for any $a,b,c\in G.$ An ordered cancellative
monoid can be embedded into an ordered abelian group $\Gamma$ in
the well-known way by introducing formal fractions $\frac{a}{b}$
for $a,b\in G.$ Then an \m-valuation $v$ from $R$ to
$T(G)=G\cup\{0\}$ is essentially the same thing as an \m-valuation
from $R$ to $\Gamma\cup\{0\}.$ For this reason, we extend the
notion of ``valuation" from above as follows.

\begin{defn}\label{defn2.3}
A \bfem{valuation} on a semiring $R$ is an \m-valuation $v: R\to
G\cup\{0\}$ with $G$ a cancellative monoid.
\end{defn}

\m-valuations on rings have been studied in \cite{HV}, and then by
D. Zhang \cite{Z}.

If $R$ is a \bfem{ring}, an \m-valuation $v: R\to M $ can
\bfem{never} be strict, since we have an element $-1\in R$ with
$1+(-1)=0,$ from which for $v$ strict it would follow that
$0_M=v(0)=\max(v(1),v(-1));$ hence $v(1)=0_M,$ a contradiction to
axiom V2. But for $R$ a semiring there may exist interesting
strict \m-valuations, even with values in a group.

\begin{example}\label{example2.2}
Let $T$ be a \bfem{preprime} in a ring $R,$ by which we simply
mean that $T$ is a sub-semiring of $R$ \ $(T+T\subset T, \ T\cdot
T\subset T,\ 0\in T, 1\in T).$ \{We do not exclude the case $-1\in
T$ (``improper preprime") but these will not matter.\}

We say that a valuation $v: R\to M$ is $T$-\bfem{convex} if the
restriction $v | T: T\to M$ is strict. As is well-known, if
$T=\sum R^2$ (and $M\setminus\{0\}$ is a group) the $T$-convex
valuations are just the real valuations on $R.$ (A valuation $v:
R\to \Gamma\cup\{0\}$ is called \bfem{real} if the residue class
field $k(v)$ is formally real.) See \cite{KZ1}, \S5 for $T$ a
preordering, and \S2 for $T=\sum R^2.$
\end{example}

 The entire paper
\cite{KZ1} witnesses the importance of $T$-convex valuations for
$T$ a preordering.

If $R$ is a ring, every \m-valuation on $R$ is strong. This can be
seen by the same argument as is well-known for valuations on
fields.

Semirings, even semifields, may admit valuations which are not
strong.

\begin{examp}\label{examp2.7}
Let $F$ be a totally ordered field, and $R:=\{x\in F|x \geq 0\}$
the subsemifield of nonnegative elements. Further let
$\Gamma:=\{x\in F|x > 0\},$ viewed as a totally ordered group, and
$M:=\{0\}\cup\Gamma$ the associated bipotent semifield. The map
$v: R\to M$ with $v(0)=0,$ $v(a)=\frac{1}{a}$ for $a\ne0$, is a
valuation on $R,$ which is not strong.
\end{examp}

\begin{prop}\label{prop2.4}
{}\quad

\begin{enumerate}
\item[a)] If $v: R\to M$ is an \m-valuation and $M$ is a bipotent
semidomain, then $v^{-1}(0)$ is a prime ideal of $R$ (i.e., an
ideal of $R,$ whose complement in $R$ is closed  under
multiplication).

\item[b)] If  $v$ is   strong, then, for any $x\in R$ and $z\in
v^{-1}(0),$
\begin{equation}\label{2.1}
v(x+z)=v(x).\end{equation}
\end{enumerate}
\end{prop}

\begin{proof} a): If $v(x)=0,$\ $v(y)=0,$ then
$$v(x+y)\le\max(v(x),v(y))=0;$$
hence $v(x+y)=0.$ Thus $v^{-1}(0)$ is closed under addition. If
$x\in R,$\ $z\in v^{-1}(0),$ then $v(xz)=v(x)v(z)=0.$ Thus
$v^{-1}(0)$ is closed under multiplication by elements in $R.$ If
$v(x)>0,$ \ $v(z)>0$, then $v(xz)=v(x)v(z)>0.$ Thus $R\setminus
v^{-1}(0)$ is closed under multiplication.

b): We have $v(x+z)\le \max(v(x),v(z))=v(x).$ Assume  that $v$ is
strong. If $v(x)\ne0$ we have
$$
v(x+z)=\max(v(x),v(z))=v(x).
$$
\end{proof}

If $v: R\to M$ is an arbitrary \m-valuation, then it is still
obvious that $v^{-1}(0)$ is an ideal of~$R.$

\begin{defn}\label{defn2.5}
We call the   ideal $v^{-1}(0)$ the \bfem{support} of the
\m-valuation $v,$ and write $v^{-1}(0)=\supp(v).$ We call the
support of $v$ \bfem{insensitive}, if the equality \eqref{2.1}
above holds for any $x\in R$ and $z\in\supp(v),$ \bfem{sensitive}
otherwise.
\end{defn}

\propref{prop2.4}.b tells us that $\supp(v)$ is insensitive if $v$
is strong. In particular, this holds if $R$ is a ring.

\begin{example}\label{example2.7}
Let $\Gamma$ be an ordered abelian group and $H$ is a convex
proper subgroup. Let $\mathfrak a :=\{g\in \Gamma\mid g>
H\}\cup\{0\}.$ We regard $\Gamma\cup\{0\}$ as a bipotent semifield
(cf. \S1), and define a subsemiring $M$ of $\Gamma\cup\{0\}$ by
$$M:=H\cup\mathfrak a .$$
Notice that we have $H\cdot\mathfrak a \subset\mathfrak a ,$\
$\mathfrak a \cdot \mathfrak a \subset\mathfrak a ,$ and
$\mathfrak a +\mathfrak a \subset \mathfrak a .$ Thus $M$ is
indeed a subsemiring of $\Gamma\cup\{0\},$ and $\mathfrak a $ is
an ideal of $M$. We define a map  $v: M\to H\cup\{0\}$ by setting
$v(x)=x$ if $x\in H,$ and $v(x)=0$ if $x\in\mathfrak a .$ It is
easily checked that $v$ fulfills the axioms V1--V3 and moreover
has  the following ``bipotency":
$$ \text{If } \ a,b \in M \  \text{ and } \ v(a) \neq v(b),  \ \text{ then } \ v(a+b) \in \{ v(a), v(b)\}.$$
 But the support $\mathfrak a$ of $v$
is sensitive: For $x\in H,$ \ $z\in \mathfrak a $ and $z\ne0,$ we
have $v(x)>0,$\ $v(z)=0,$\ $x+z=z;$ hence $v(x+z)=0.$
\end{example}

We switch over to the problem of ``comparing" different
\m-valuations on the same semi-ring~$R.$

\begin{defn}\label{defn2.11}
Let $v: R\to M$ and $w: R\to N$ be \m-valuations.
\begin{enumerate}\item[a)] We say that $v$ \bfem{dominates} $w,$
if for any $a,b\in R$
$$v(a)\le v(b)\Rightarrow w(a)\le w(b).$$
\item[b)] If $v$ dominates $w$  and $v$ is surjective, there
clearly exists a unique map \\ $\gamma: M\to N$ with
$w=\gamma\circ v.$ We denote this map $\gamma$ by $\gamma_{w,v}.$
\end{enumerate}
\end{defn}

Clearly, $\gamma_{w,v}$ is multiplicative and sends $0$ to $0$ and
$1$ to $1$. $\gamma_{w,v}$ is also order-preserving and hence is a
homomorphism from the bipotent semiring $M$ to $N.$

\begin{prop}\label {prop2.12}
Assume that $M,N$ are bipotent semirings and $v:R\to M$ is a
surjective \m-valuation.
\begin{enumerate}\item[a)] The \m-valuations $w: R\to N$
dominated by $v$ correspond uniquely with the homomorphisms
$\gamma: M\to N$ via $w=\gamma\circ v,$ $\gamma=\gamma_{w,v}.$
\item[b)] If $v$ has one of the properties ``strict",  or ``strong",  and dominates $w$, then $w$ has the same
property.
\end{enumerate}
\end{prop}

\begin{proof}
If $w$ is an \m-valuation dominated by $v$ then we know already
that $\gm := \gm_{w,v}$ is a homomorphism and $w = \gm \circ v$.
On the other hand, given a homomorphism $\gm: M \to N$, clearly
$\gm \circ v$ is an \m-valuation, and $\gm \circ v$ inherits from
$v$ each of the properties ``strict" and ``strong".
%
%
\end{proof}

We mention that for strong \m-valuations the dominance condition
in Definition \ref{defn2.11} can be weakened.

\begin{prop}\label{prop2.122}
Assume that $v:R \to M$ and $w:R \to N$ are strong \m-valuations
and that
$$ \forall a,b \in R: \qquad v(a) = v(b) \ \Rightarrow \ w(a) = w(b).$$

Then $v$ dominates $w$.
\end{prop}

\begin{proof}
Let $a,b \in R$ and assume that $v(a) \leq v(b)$. If  $v(a) <
v(b)$ then  $v(a+ b) =  v(b)$, hence $w(a+b) = w(b)$. It follows
that $w(a) \leq w(b)$ since $w(a) > w(b)$ would imply $w(a+b) =
w(a).$ Thus $w(a) \leq w(b)$ in both cases.
\end{proof}

\section{Supertropical semirings}\label{sec:3}

\begin{defn}\label{defn3.1} A \bfem{semiring with idempotent} is
a pair $(R,e)$ consisting of a semiring $R$ and a \bfem{central}
idempotent $e.$ \{For the moment $R$ is allowed to be
noncommutative.\}
\end{defn}

We then have an endomorphism $\nu: R\to R$ (which usually does not
map 1 to 1) defined by $\nu(a)=ea.$ It obeys the rules

\begin{equation}\label{3.1} \nu\circ \nu=\nu,
\end{equation}
\begin{equation}\label{3.2} a\nu(b)=\nu(a)b=\nu(ab).
\end{equation}

\noindent Conversely, if a pair $(R,\nu)$ is given consisting of a
semiring $R$ and an endomorphism $\nu$ (not necessarily
$\nu(1)=1),$ such that \eqref{3.1}, \eqref{3.2} hold, then
$e:=\nu(1)$ is a central idempotent of $R$ and $\nu(a)=ea$ for
every $a\in R.$

Thus such pairs $(R,\nu)$ are the same objects as semirings with
idempotents.

\begin{defn}\label{defn3.2} A \bfem{semiring with ghosts} is
a  semiring with idempotent $ (R,e)$ such that the following axiom
holds  $(\nu(a):=ea)$
\begin{equation}\label{3.3}
\nu(a)=\nu(b)\quad\Rightarrow \quad a+b=\nu(a).
\end{equation}
\end{defn}

\begin{remark}\label{rem3.3}
This axiom implies  that  $ea=e(a+b)=ea+eb$ if $\nu(a)=\nu(b).$ We
\bfem{do not} want to demand that then $eb=0.$ Usually, $(R,+)$
will be a highly non-cancellative abelian semigroup.
\end{remark}

\begin{terminology}\label{term3.4}
If  $(R,e)$ is a semiring with ghosts, then $\nu : x\mapsto ex,$\
$R\to R$ is called the \bfem{ghost map} of $(R,e).$ The idea is
that every $x\in R$ has an associated ``ghost" $\nu(x),$ which is
thought of to be somehow ``near" to the zero element $0$ of $R,$
without necessarily being 0 itself. \{That will happen for all
$x\in R$ only if $e=0.$\} We call $eR$ the \bfem{ghost ideal} of
$(R,e).$
\end{terminology}

Now observe that, if $(R,e)$ is a semiring with ghosts, the
idempotent $e$ is determined by the semiring $R$ above, namely
$$e=1+1.$$
Thus we may suppress the idempotent $e$ in the notation of a
semiring with ghosts and redefine these objects as follows.

\begin{defn}\label{defn3.5}
A semiring $R$ is called a \bfem{semiring with ghosts} if
\begin{equation}1+1=1+1+1+1\tag{3.3$'$}
\end{equation}
and for all $a,b\in R$
\begin{equation}a+a=b+b\quad\Rightarrow \quad  a+b=a+a. \tag{3.3$''$}
\end{equation}
\end{defn}

\begin{remark}
If {\rm(3.3)$'$} holds then $e:=1+1$ is a central idempotent of
$R.$ Passing from $R$ to $(R,e)=(R,1+1),$ we see that {\rm(3.3
$''$)} is the previous axiom \eqref{3.3}. Notice also that
{\rm(3.3$''$)} implies that $1+1+1=1+1.$ (Take $a=1,$ $b=e.$)
Thus, $m1=1+1$ for all natural numbers $m\ge 2.$
\end{remark}

\begin{terminology}\label{term3.6}
If  $R$ is a semiring with ghosts, we write $e=e_R $ and
$\nu=\nu_R$ if necessary. We also introduce the notation
\begin{align*}
&\mathcal T:=\mathcal T(R):=R\setminus Re,\\
&\mathcal G:=\mathcal G(R):=Re\setminus\{0\},\\
&\mathcal G_0:= \mathcal G\cup\{0\}=Re.
\end{align*}
We call the elements of $\mathcal T$ the \bfem{tangible elements}
of $R$ and the elements of $\mathcal G$ the \bfem{ghost elements}
of $R.$ We do not exclude the case that $\mathcal T$ is empty,
i.e., $e=1.$ In this case $R$ is called a \bfem{ghost semiring}.
\end{terminology}

The ghost ideal $\mathcal G_0=eR$ of $R$ is itself a semiring with
ghosts, in fact, a ghost semiring. It has the property $a+a=a$ for
every $a\in Re,$ as follows from \eqref{3.3}. \{Some people call a
semiring $T$ with $a+a=a$ for every $a\in T$ an ``idempotent
semiring".\}

We mention a consequence of axiom \eqref{3.3} for the ghost map
$\nu: R\to Re,$\ $\nu(x):=ex.$

\begin{remark}\label{remark3.7} If $R$ is a semiring with ghosts,
then, for any $x\in R,$
$$\nu(x)=0\quad\Leftrightarrow\quad x=0.$$
\end{remark}

\begin{proof} $(\Leftarrow)$: evident.

$(\Rightarrow)$: We have $\nu(x)=0=\nu(0);$ hence by \eqref{3.3}
$x=x+0=\nu(0)=0.$
\end{proof}

We are ready for the central definition of the section.

\begin{defn}\label{defn3.8}
A semiring $R$ is called \bfem{supertropical} if $R$ is a semiring
with ghosts and
\begin{equation}\label{3.4}
\forall a,b\in R:\ a+a\ne b+b\quad\Rightarrow \quad
a+b\in\{a,b\}.\end{equation} In other terms, every pair $(a,b)$ in
$R$ with $ea\ne eb$ is bipotent.
\end{defn}

\begin{remarks}\label{remark3.9} $ $ 

\begin{enumerate} \eroman
    \item

It follows that then $\mathcal G(R)_0=Re$ is a bipotent
semidomain. Indeed, if $a,b$ are different elements of $\mathcal
G(R)$, then $a=ea\ne b=eb;$ hence $a+b\in\{a,b\}$ by axiom
\eqref{3.4}. If $a=0$ or $b=0$, this trivially is also true. If
$a=b$ then $a+b=ea=a.$ Thus $a+b\in\{a,b\}$ for any $a,b\in
\mathcal G(R)_0.$ The set $\mathcal G(R)$ is either empty (the
case $1+1=0)$ or $\mathcal G(R)$ is an ordered monoid under the
multiplication of $R,$ as explained in \S$1$.

\item The supertropical semirings without tangible elements are just
the bipotent semirings.

\item Every subsemiring  of a
supertropical semiring is again supertropical.
\end{enumerate}
\end{remarks}

\begin{thm}\label{thm3.10} Let $R$ be a supertropical
semiring, $e:=e_R,$ \ $\mathcal G:=\mathcal G(R).$ Then the
addition on~$R$ is determined by the multiplication on~$R$ and the
ordering on the multiplicative submonoid~$\mathcal G$ of $R,$ in
case  $\mathcal G\ne\emptyset,$ as follows. For any $a,b\in R$
$$a+b=\begin{cases} a\quad & \text{if}\quad  ea>eb,\\
 b\quad & \text{if}\quad  ea<eb,\\
  ea\quad & \text{if}\quad  ea=eb,\end{cases}$$
  If $\mathcal G=\emptyset$ then $a+b=0$ for any $a,b\in R.$
  \end{thm}

  \begin{proof}
  We may assume that $ea\ge eb.$ If $ea=eb,$ axiom \eqref{3.3}
  tells us that $a+b=ea.$ Assume now that $ea>eb.$ By definition
  of the ordering on $eR$ (cf. \S1), we have
  $$e(a+b)=ea+eb=ea.$$
  By axiom \eqref{3.4}, $a+b=a$ or $a+b=b.$

Suppose that $a+b=b.$ Then $e(a+b)=eb.$ Since $ea\ne eb,$ this is
a contradiction. We conclude that $a+b=a.$
  \end{proof}

  From now on, \bfem{we always assume that our semirings are
  commutative}.

  \begin{remark}\label{rem3.12}
  If $R$ is a supertropical semiring, the ghost map
  $\nu_R: R\to eR,$\ $x\mapsto ex$ is a strict \m-valuation.
  Indeed, the axioms V1--V3 and V5 from \S2 are clearly valid for
  $\nu_R.$\end{remark}

  Thus, every supertropical semiring has a natural built-in strict
  \m-valuation.

There are important cases where $\nu_R$ is even a valuation
(cf.~Definition \ref{defn2.3}), as we explicate now.

\begin{prop}\label{prop3.12} Assume that $R$ is a supertropical
semiring and $\mathcal T(R)$ is closed under multiplication. Then
the submonoid $G:=e\mathcal T(R)$ of $\mathcal G(R)$ is
cancellative. (N.B. We have $e\mathcal T(R)\subset\mathcal G(R)$
by Remark \ref{remark3.7}.)
\end{prop}

\begin{proof} Let $a,b,c\in\mathcal T(R)$ be given with
$(ea)(ec)=(eb)(ec),$ i.e., $eac=ebc.$ Suppose that $ea\ne eb,$ say
$ea<eb.$ Then \thmref{thm3.10} tells us that $a+b=b $  and
$ac+bc=ebc.$ By assumption, $bc\in\mathcal T(R)$; hence $bc\ne
ebc.$ But the first equation gives $ac+bc=bc,$ a contradiction.
Thus $ea=eb.$\end{proof}

In the situation of this proposition we may omit the part
$\mathcal G(R)\setminus G$, consisting of ``useless" ghosts, in
the semiring $R$, and then obtain a ``supertropical domain" $U:=
\mathcal T(R)\cup G\cup\{0\},$ as defined below, whose ghost map
$\nu_U:= U\to G\cup\{0\}$ is a \emph{surjective strict valuation}.

  \begin{defn}\label{defn3.13}
  Let $M$ be a bipotent semiring and $R$ a supertropical semiring.

\begin{enumerate}
    \item[a)]
    We say that the semiring $M$ is  \bfem{cancellative}
    if for any $x,y,z \in M$
    $$ xz = yz, \ z \neq 0 \ \Rightarrow \ x = y.$$
This means that $M$ is a bipotent semidomain (cf.~Definition
\ref{defn1.4}) and the multiplicative monoid $M \setminus \{ 0 \}$
is cancellative.

   \item[b)]
   We call $R$ a \bfem{supertropical predomain}, if
  $\mathcal T(R)=R\setminus eR$ is not empty (i.e., $e\ne1)$ and is closed
  under multiplication, and moreover  $eR$ is a cancellative bipotent
  semidomain.

 \item[c)]
  We call $R$ a \bfem{supertropical domain}, if
  $  \mathcal T(R)$ is not empty and is closed under
  multiplication, and $R$ maps $\mathcal T(R)$ onto
  $\mathcal
  G(R).$
  \end{enumerate}
  \end{defn}

Notice that the last condition in Definition \ref{defn3.13}.c
implies that $\mathcal G(R)$ is a cancellative monoid
(\propref{prop3.12}). Thus a supertropical domain is a
supertropical predomain.

Looking again at \thmref{thm3.10}, we see that a way is opened up
to construct supertropical predomains and domains. First notice
that the theorem implies the following

 \begin{remark}\label{rem3.13}
If $R$ is a supertropical predomain, we have for every $a\in
\mathcal T(R)$ and $x\in \mathcal G(R)$ the multiplication rule
$$ax=v(a)x$$
with $v:=\nu_R\mid \mathcal T(R).$ Thus the multiplication on
$$R=\mathcal T(R)\ \dot\cup\ \mathcal G(R)\ \dot\cup\ \{0\}$$
is completely determined by the triple $(\mathcal T(R),\mathcal
G(R),v).$ We write $v=v_R.$
  \end{remark}

  \begin{construction}\label{constr3.14}
  Conversely, let a triple $(\mathcal T,\mathcal G,v)$ be given with $\mathcal T$ a monoid,
  $\mathcal G$ an ordered cancellative
   monoid and $v:\mathcal T\to \mathcal G$ a monoid homomorphism. We
  define a semiring $R$ as follows. As a set
  $$R=\mathcal T\ \dot\cup\  \mathcal G\ \dot\cup\ \{0\}.$$
  The multiplication on $R$ will extend the given multiplications
  on $\mathcal T$ and $\mathcal G.$ If $a\in \mathcal T,$\ $x\in \mathcal G,$ we decree that
  $$a\cdot x=x\cdot a:=v(a)x.$$
  Finally, $0\cdot z=z\cdot 0:=0\quad\text{for all}\quad z\in R.$

  The addition on $R$ extends the addition on $\mathcal G\cup\{0\}$ as
  the bipotent semiring corresponding to the ordered monoid $\mathcal G,$
  as explained in \S$1$. For $x,y\in \mathcal T$ we decree
  $$x+y:=\begin{cases} x&\ \text{if}\quad v(x)>v(y),\\
  y&\ \text{if}\quad v(x)<v(y),\\
  v(x)&\ \text{if}\quad  v(x)=v(y).\end{cases}$$
  Finally, for $x\in \mathcal T$ and $y\in \mathcal G\cup\{0\}$
  $$x+y=y+x:=\begin{cases} x&\  \text{if}\quad v(x)> y, \\
  y&\ \text{if}\quad v(x)\le y.\end{cases}$$

  It now can be checked in a straightforward way\footnote{Alternatively consult \cite[\S 3]{IKR} (as soon as available),
   where a
  detailed proof of a more general statement is given.} that $R$ is a
  supertropical predomain with $\mathcal T(R)=\mathcal T,$\ $\mathcal G(R)= \mathcal G,$\ $v_R=v.$
  Thus we have gained a description of all supertropical
  predomains $R$ by triples $(\mathcal T,\mathcal G,v)$ as above. We write
 $$
  R=\STR(\mathcal T,\mathcal G,v)$$
  \{$\STR$ = ``supertropical"\}. Notice that in this semiring $R$
  every pair $(x,y)\in R^2$ is bipotent except the pairs $(a,b)$
  with $a\in \mathcal T,$ \ $b\in \mathcal T$ and $v(a)=v(b).$ If
  $v$ is onto, then $R$ is a supertropical domain.
  \end{construction}

  \begin{defn}\label{defn3.15}
  A semiring $R$ is called a \bfem{supertropical semifield}, if
  $R$ is a supertropical domain, and every $x\in \mathcal T(R)$ is
  invertible; hence both $\mathcal T(R)$ and $\mathcal G(R)$ are groups under
  multiplication.\end{defn}

  We write down primordial examples of supertropical domains and
  semifields (cf. \cite{I}, \cite{IzhakianRowen2007SuperTropical}). Other examples will come
  up in \S$4$.

  \begin{examples}\label{examps3.16} Let $\mathcal G$ be an ordered cancellative monoid.
  This given us the supertropical domain (cf.~Construction \ref{constr3.14})
$$D(\mathcal G):=\STR(\mathcal G, \mathcal G,\id_{\mathcal G}).$$
$D(\mathcal G)$ is a supertropical semifield iff $\mathcal G$ is
an ordered abelian group.

We come closer to the objects and notations of usual tropical
algebra if we take here for $\mathcal G$ ordered monoids in
\bfem{additive} notation, $\mathcal G=(\mathcal G,+),$ e.g.,
$\mathcal G=\mathbb R,$ $\mathbb R_{>0},$ $\mathbb N,$ $\mathbb
Z,$ $\mathbb Q$ with the usual addition. $D(\mathcal G)$ contains
the set~$\mathcal G.$ For every $a\in  \mathcal G$ there is an
element $a^\nu$ in $D(\mathcal G)$ (read ``a-ghost"), and
$$\mathcal G^\nu:=\{a^\nu \ | \ a\in \mathcal G\}$$
is a copy of the additive monoid $\mathcal G$ disjoint from
$\mathcal G.$ The zero element of the semiring $D(\mathcal G)$ is
now written $-\infty.$ Thus
$$D(\mathcal G)=\mathcal G\ \dot  \cup\ \mathcal G^\nu\dot\cup\ \{-\infty\}.$$
Denoting addition and multiplication of the semiring $D\mathcal (
G)$ by $\oplus$ and $\odot,$ we have the following rules. For any
$x\in D(\mathcal G),$\ $a\in \mathcal G,$\ $b\in \mathcal G,$
\begin{align*}
 -\iy\oplus x&=x\oplus-\infty=x,\\
 a\oplus b& =\max(a,b),\quad\text{if}\ a\ne b,\\
 a\oplus a&=a^\nu,\\
a^\nu\oplus b^\nu&=\max(a,b)^\nu,\\
a\oplus b^\nu&=a,\quad \text{if}\ a >b,\\
a\oplus b^\nu&=b^\nu,\quad\text{if}\ a \le b,\\
-\infty\odot x&=x\odot-\infty=-\infty,\\
a\odot b&=a+b,\\
a^\nu\odot b&=a\odot b^\nu=a^\nu\odot b^\nu=(a+b)^\nu.
\end{align*}
  \end{examples}

  In the case $\mathcal G=(\mathbb R,+)$ these rules can already be found
  in \cite{I}. There also motivation  is
  given for their use in tropical algebra and tropical geometry.

  We now only say that the semiring $D(\mathcal G)$ associated to an
  additive ordered cancellative
   monoid $\mathcal G$ should be compared with the max-plus
  algebra $T(\mathcal G)=\mathcal G\cup\{-\infty\}$ introduced in \S1. The ghost
  ideal $ \mathcal G^\nu\cup\{-\infty\}$ of $D( \mathcal G)$ is a copy of $T(\mathcal G).$

  \section{Supervaluations}\label{sec:4}

  In this section $R$ is always a (commutative) semiring. Usually
  the
  letters $U,V$ denote supertropical (commutative) semirings. If
  $U$ is any such semiring, the idempotent $e_U=1_U+1_U$ will be
  often simply denoted by the letter ``$e$", regardless of which
  supertropical semiring is under consideration (as we write
  $0_U=0,$\ $1_U=1).$


  \begin{defn}\label{defn4.1} $ $
  \begin{enumerate}
  \item[a)]

     A \bfem{supervaluation} on $R$ is a map $\varphi: R\to U$
  from $R$ to a supertropical semiring $U$ with the following
  properties.
 \begin{alignat*}{2}
&SV1:\ &&\varphi(0)=0,\\
&SV2:\ &&\varphi(1)=1,\\
&SV3:\ &&\forall a,b\in R: \varphi(ab)=\varphi(a)\varphi(b),\\
&SV4:\ &&\forall a,b\in R: e\varphi(a+b)\le
e(\varphi(a)+\varphi(b))\quad
[=\max(e\varphi(a),e\varphi(b))].\end{alignat*}

\item[b)]
 If $\varphi: R\to U$ is a supervaluation, then the map
$$v: R\to eU,\qquad v(a):=e\varphi(a)$$
is clearly an \m-valuation. We denote this \m-valuation $v$ by
$e_U\varphi $ (or simply by $e\varphi )$, and we say that
$\varphi$ \bfem{covers} the \m-valuation $e_U\varphi=v.$


\item[c)]
 We say that a supervaluation $\varphi: R\to U$ is
\bfem{tangible}, if $\varphi(R)\subset \mathcal T(U)\cup\{0\}$,
and we say that $\varphi $ is   \bfem{ghost}  if
$\varphi(R)\subset eU.$
\end{enumerate}
\end{defn}
\noindent {N.B.}\ A ghost supervaluation $\varphi: R\to U$ is
nothing other than an \m-valuation, after replacing the target $U$
by $eU.$


\begin{prop}\label{prop4.2}
Assume that $\varphi:R\to U$ is a supervaluation and $v: R\to
e_UU=: M$ is the \m-valuation $e_U\varphi$ covered by $\varphi.$
Then
$$U':=\varphi(R)\cup e\varphi(R) $$
is a subsemiring of $U.$ The semiring $U'$ is again supertropical
and $e_{U'}=e_U  (=e).$\end{prop}

\begin{proof} The set
 $v(R)$ is a multiplicative submonoid  of the bipotent semiring
$M$; hence is itself a bipotent semiring. In particular, $v(R)$ is
closed under addition. If $a,b\in R$ are given with $v(a)\le
v(b),$ then either $v(a)< v(b),$ in which case
$$a+b=b,\quad v(a)+b=b,\quad a+v(b)=v(b),$$
or $v(a)=v(b), $ in which case
$$a+b=v(a)+b=a+v(b)=v(a).$$
This proves that $U'+U'\subset U'.$ Clearly $0\in U',$\ $1\in U'$
and $U'\cdot U'\subset U'.$ Thus $U'$ is a subsemiring of $U.$ As
stated above (Remark \ref{remark3.9}.iii), every subsemiring of a
supertropical semiring is again supertropical. Thus $U'$ is
supertropical.
\end{proof}

\begin{defn}\label{defn4.3}
We say that the supervaluation $\varphi: R\to U$ is
\bfem{surjective} if $U'=U.$ We say that $\varphi$ is
\bfem{tangibly surjective} if $\varphi(R)\supset  \mathcal
T(U).$\end{defn}

\begin{remark}\label{rem4.4} If $\varphi: R\to U$ is any
supervaluation, then, replacing $U$ by $U'=\varphi(R)\cup
e\varphi(R),$ we obtain a surjective supervaluation. If we only
replace $U$ by $\varphi(R)\cup (eU)$, which is again a subsemiring
of $U,$ we obtain a tangibly surjective
supervaluation.\end{remark}

Thus, whenever necessary we may retreat to tangibly surjective or
even surjective supervaluations without loss of generality.

Recall that an \m-valuation $v: R\to M$ is called a
 valuation, if the bipotent semiring $M$ is   cancellative
(cf.~Definition 2.3, Definition 3.14.a). Every
 valuation can be covered by a tangible supervaluation, as the
following easy but important construction shows.

\begin{example}\label{examp4.5} Let $v: R\to M$ be a
valuation, and let $\mathfrak q :=v^{-1}(0)$ denote the support
of~$v.$ We then have a monoid homomorphism
$$R\setminus \mathfrak q \to M\setminus\{0\},\quad a\mapsto v(a),$$
which we denote again by $v.$ Let
$$U:=\STR(R\setminus \mathfrak q ,M\setminus\{0\},v),$$
the supertropical predomain given by the triple
$(R\setminus\mathfrak q ,M\setminus \{0\},v),$ as explained in
Construction~\ref{constr3.14}. Thus, as a set,
$$U=(R\setminus\mathfrak q )\ \dot\cup\ M.$$
We have $e=1_M,$ $e\cdot a=v(a)$ for  $a\in R\setminus \mathfrak q
.$ The multiplication on $U$ restricts to the given
multiplications on $ R\setminus \mathfrak q$ and on $M$, and
$a\cdot x=x\cdot a=v(a)x$ for $ a\in  R\setminus \mathfrak q,$
$x\in M.$ The addition on $U$ is determined by $e$ and the
multiplication in the usual way (cf. \thmref{thm3.10}). In
particular, for  $a,b\in R\setminus\mathfrak q ,$ we have
$$a+b=\begin{cases} a &\quad\text{if}\ v(a)>v(b),\\
 b &\quad\text{if}\ v(a)<v(b),\\
  v(a) &\quad\text{if}\ v(a)=v(b).\end{cases}$$

  Now define a map $\varphi: R\to U$ by
  $$\varphi(a):=\begin{cases}  a &\quad\text{if}\ a\in R\setminus\mathfrak q, \\
  0&\quad\text{if}\ a\in  \mathfrak q.\end{cases}$$
  One checks immediately that $\varphi$ obeys the rules SV1--SV3.
  If $a\in R\setminus\mathfrak     q ,$ then
  $$e_U\varphi(a)=1_M\cdot v(a)=v(a),$$
  and for $x\in\mathfrak q ,$ we have
  $$e_U\varphi(a)=e_U\cdot 0=0=v(a)$$
  also. Thus SV4 holds, and $\varphi$ is a supervaluation covering
  $v.$

  By construction $\varphi$ is tangible and tangibly surjective.
  If $v$ is surjective then $\varphi$ is surjective.

\end{example}

\begin{defn}\label{defn4.6} We denote the supertropical ring just
constructed by $U(v)$ and the supervaluation $\varphi$ just
constructed by $\varphi_v.$ Later we will call $\varphi_v: R\to
U(v)$ \bfem{the initial cover  of}~$v,$ cf.~Definition
\ref{defn5.15}.\end{defn}

Notice that $U(v)$ is a supertropical domain iff $v$ is
surjective, and that in this case the supervaluation $\varphi_v$
is surjective.

\begin{remark}\label{rem4.7}
The supertropical predomain $U(v)$ just constructed deviates
strongly in its nature from the supertopical domain $D(\mathcal
G)$ for $\mathcal G$ an ordered monoid studied in
Examples~\ref{examps3.16}. While for $U=D(\mathcal G)$ the
restriction
$$\nu_U\mid \mathcal T(U): \mathcal T(U)\to \mathcal G(U)$$ of the ghost
map $\nu_U$ is bijective, for $U=U(v)$ this map usually has big
fibers.\end{remark}

\section{Dominance and transmissions}\label{sec:5}

As before now all semirings are assumed to be commutative. $R$ is
any semiring, and $ U,V$ are bipotent semirings.

\begin{defn}\label{defn5.1} If $\varphi: R\to U$ and $\psi: R\to
V$ are supervaluations, we say that $\varphi$ \bfem{dominates}
$\psi$, and write $\varphi\ge \psi,$ if for any $a,b\in R$ the
following holds.
\begin{alignat*}{3}
&D1.\quad && \varphi(a)=\varphi(b)&&\Rightarrow \psi(a)=\psi(b),\\
&D2.\quad && e\varphi(a)\le e\varphi(b)&&\Rightarrow e\psi(a)\le e\psi(b),\\
&D3.\quad && \quad\varphi(a)\in eU &&\Rightarrow \psi(a)\in eV.
\end{alignat*}
Notice that D3 can be also phrased as follows:
$$\varphi(a)=e\varphi(a)\quad\Rightarrow \quad  \psi(a)=e\psi(a).$$
\end{defn}

\begin{lem}\label{lem5.2} Let $\varphi: R\to U$ and $\psi: R\to V$
be supervaluations. Assume that $\varphi$ dominates $\psi$, and
also (without essential loss of generality) that $\varphi$ is
surjective. Then there exists a unique map $\alpha: U\to V$ with
$\psi=\alpha\circ \varphi$ and
$$\forall x\in U: \alpha(e_Ux)=e_V\alpha(x)$$
(i.e., $\alpha\circ\nu _U=\nu_V\circ\alpha).$
\end{lem}

\begin{proof} By D1 and D2 we have a unique well-defined map
$\beta: \varphi(R)\to\psi(R)$ with $\beta(\varphi(a))=\psi(a)$ for
all $a\in R$ and a unique well-defined map $\gamma: e\varphi(R)\to
e\psi(R)$ with $\gamma(e\varphi(a))=e\psi(a)$ for all $a\in R. $
Now $U=\varphi(R)\cup e\varphi(R)$, since $\varphi$ is assumed to
be surjective. Suppose that $x\in\varphi(R)\cap e\varphi(R).$ Then
$x=\varphi(a)$ for some $a\in R,$ and $x=ex=e\varphi(a).$ By axiom
D3 we conclude that $\psi(a)=e\psi(a).$ Thus $\beta(x)=\gamma(x).$
This proves that we have a unique well-defined map $\alpha: U\to
V$ with $\alpha(x)=\beta(x)$ for $x\in\varphi(R)$ and
$\alpha(y)=\gamma(y)$ for $y\in e\varphi (R).$ We have
$\alpha(\varphi(a))=\psi(a),$ i.e., $\psi=\alpha\circ\varphi.$
Moreover, for any $a\in R,$
$\alpha(e_U\varphi(a))=\gamma(e_U\varphi(a))=e_V\psi(a).$\end{proof}

We record that in this proof we did not use the full strength of
D2 but only the weaker rule that $e\varphi(a)=e\varphi(b)$ implies
$e\psi(a)=e\psi(b).$

\begin{defn}\label{defn5.3} Assume that $U$ and $V$ are
supertropical semirings.

\begin{enumerate}\item[a)] If $\alpha$ is a map from $U$ to $V$ with
$\alpha(eU)\subset eV,$ we say that $\alpha$ \bfem{covers} the map
$\gamma: eU\to eV$ obtained from $\alpha$ by restriction, and we
write $\gamma=\alpha^\nu.$ We also say that $\gamma$ is the
\bfem{ghost part} of $\alpha.$

\item[b)] Assume that $\varphi: R\to U$ is a surjective
supervaluation and $\psi: R\to V$ is a supervaluation dominated by
$\varphi.$ Then we call the map $\alpha$ occurring in
\lemref{lem5.2}, which is clearly unique, the \bfem{transmission
from} $\varphi$ \bfem{to} $\psi,$ and we denote this map by
$\alpha_{\psi,\varphi}.$ Clearly, $\alpha_{\psi,\varphi}$ covers
the map $\gamma_{w,v}$ connecting the surjective \m-valuation $v:=
e\varphi: R\to eU$ to the \m-valuation $w:=e\psi: R\to eV$
introduced in Definition \ref{defn2.11}.
\end{enumerate}
 \end{defn}

 \begin{thm}\label{thm5.4}
 Let $\varphi: R\to U$ be a surjective supervaluation and $\psi:
 R\to V$ a supervaluation dominated by $\varphi.$ The transmission
 $\alpha:=\alpha_{\psi,\varphi}$ obeys the following rules:
\begin{alignat*}{2}
&TM1: \quad&& \alpha(0)=0,\\
&TM2: \quad &&\alpha(1)=1,\\
 &TM3: \quad &&\forall x,y\in U:\quad
\alpha(xy)=\alpha(x)\alpha(y),\\
&TM4:\quad&&\alpha(e_U)=e_V,\\
&TM5:\quad &&\forall x,y\in eU: \quad
\alpha(x+y)=\alpha(x)+\alpha(y).
 \end{alignat*}
 \end{thm}

 \begin{proof}

 TM1, TM2, and TM4 are obtained from the construction of $\alpha$
in the proof of  \lemref{lem5.2}. This construction tells us also
that $\alpha$ sends  $e U$ to $eV$. Using (again) that $U =
\varphi (R) \cup e\varphi (R)$, we check  easily that $TM3$ holds.
The rule D2 (in its full strength) tells us  that the map $\gamma:
eU \to eV$, obtained from $\alpha$ by restriction, is order
preserving. This is TM5.
\end{proof}

\begin{defn}\label{defn5.5a} If $U$ and $V$ are supertropical
semirings, we call any map $\alpha: U\to V$ which the rules
TM1--TM5, a \bfem{transmissive map} from $U$ to $V.$
\end{defn}


The axioms TM1-TM5 tell us that a transmissive map $\alpha: U \to
V$ is the same thing as a homomorphism from the monoid $(U, \cdot
\ )$ to $(V, \cdot \ )$ which restricts to a semiring homomorphism
from $eU$ to $eV$. It is evident that every homomorphism from the
semiring $U$ to $V$ is a transmissive map, but there exist quite a
few transmissive maps, which are not homomorphisms; cf. \S9 below
and \cite{IKR}.

As a converse to Lemma \ref{lem5.2} we have the following fact.

\begin{prop}\label{prop5.6}
Assume that $\varphi: R\to U$ is a supervaluation and $\alpha:
U\to V$ is a transmissive map from $U$ to a supertropical semiring
$V.$ Then $\alpha\circ \varphi: R\to V$ is again a supervaluation.
If $e \varphi$ is either ``strong" or ``strict", then
$e(\alpha\circ\varphi)$ has the same property.

\end{prop}

\begin{proof} Let $\psi: =\alpha\circ\varphi,$ $v := e \vrp$, $w:= e
\psi$.  Clearly $\psi$ inherits the properties SV1--SV3 from
$\varphi,$ since $\alpha$ obeys TM1--TM3. If $a\in R,$ then, by
TM4,
$$w(a) = e\psi(a)=e(\alpha(\varphi(a)))=\alpha(e\varphi(a)) = \al(v(a));$$
hence $w = \alpha^\nu \circ v.$ Now $\al^\nu : N \to N$ is a
semiring homomorphism, hence order preserving. Thus it is
immediate that $w$ is an \m-valuation,  and $w$ is strict or
strong if $v$  is strict or strong, respectively.
\end{proof}

\begin{remark}\label{rem5.7}
If $\varphi: R\to U$ is a surjective supervaluation
(cf.~Definition \ref{defn4.3}) and $\alpha: U\to V$ is a
surjective transmissive map, then the supervaluation
$\alpha\circ\varphi$ is again surjective. Conversely, if $\varphi:
R\to U$ and $\psi: R\to V$ are surjective supervaluations, and
$\varphi$ dominates $\psi$, then the transmission
$\alpha_{\psi,\varphi}:U\to V$ is a surjective map.\end{remark}

Combining \thmref{thm5.4}, \propref{prop5.6} and this remark, we
read off  the following facts.

\begin{schol}\label{schol5.8} Let $U,V$ be supertropical semirings
and $\varphi: R\to U$ a surjective supervaluation.
\begin{enumerate}
\item[a)] The supervaluations $\psi: R\to V$ dominated by
$\varphi$ correspond uniquely with the transmissive maps $\alpha:
U\to V$ via $\psi=\alpha\circ\varphi,$
$\alpha=\alpha_{\psi,\varphi}.$ \item[b)] If $P$ is one of the
properties ``strict" or  ``strong" and $e \varphi$ has property
$P,$ then $e \psi$ has property $P.$
\item[c)] The supervaluation $\psi$ is surjective iff the map
$\alpha$ is surjective. \item[d)] Given a semiring homomorphism
$\gamma: eU\to eV$, the supervaluation $\psi$ covers the
\m-valuation $\gamma\circ (e\varphi)$ iff $\alpha^\nu=\gamma.$
\end{enumerate}
\end{schol}
 $$\xymatrix{
    & R      \ar@{-->}[dr]^\psi  \ar[dl]_\varphi\\
           U %
            \ar[d]_{\nu_U}     \ar@{-->}[rr]_\alpha    &&      V
       \ar[d]^{\nu_V}
       \\
      eU    \ar[rr]_\gamma   &&        eV
 }$$
 \hfill\quad\qed

 \begin{examp}\label{examp5.9}
 Let $U$ be a supertropical semiring with ghost ideal $M:=eU.$
 Then, as we know, the ghost map $\nu_U:U\to M,$ $x\mapsto
 ex,$ is a strict \m-valuation on the semiring $U$ (Remark
 \ref{rem3.12}). Clearly, the identity map $\id_U:U\to U$ is a
 supervaluation covering $\nu_U.$ Assume now that $\alpha: U\to
 V$ is a transmissive map. Let $\gamma: =\alpha^\nu$ denote the
 homomorphism from $M$ to $N:=eV$ covered by $\alpha.$ Then
 $v:=\gamma\circ\nu_U=\nu_V\circ\alpha$ is  a strict valuation on
 $U$ with values in $N,$ and $\alpha:=\alpha\circ\id_U$ is a
 supervaluation on $U$ covering $v.$ Thus $\alpha $ is the
 transmission from the supervaluation $\id_U: U\to U$ to the
 supervaluation $\alpha:U\to V$ covering $v.$

\end{examp}

The example tells us in particular that every transmissive map is
the transmission between some  supervaluations. \textit{Therefore
we may and will also use the shorter term ``\textbf{transmission}"
for ``transmissive map".}

In general, a transmission does not behave additively; hence is
not a homomorphism. We now record cases where nevertheless some
additivity holds.

\begin{prop}\label{prop5.10}
Let $\alpha:U\to V$ be a transmission and $\gamma:eU\to eV$ denote
the ghost part of $\alpha,$ $\gamma=\alpha^\nu$ (which is a
semiring homomorphism). \begin{enumerate}
\item[i)] If $x,y\in U$
and $ex=ey,$ then $\alpha(x)+\alpha(y)=\alpha(x+y).$

\item[ii)]
If $x,y\in U$ and $\alpha(x)+\alpha(y)$ is tangible, then again
$\alpha(x)+\alpha(y)=\alpha(x+y).$

\item[iii)] If $\gamma$ is
injective, then $\alpha$ is a semiring homomorphism.
\end{enumerate}
\end{prop}

\begin{proof} Let $x,y\in U$ be given, and assume without loss of
generality that $ex\le ey.$ Notice that this implies
$$e\alpha(x)=\alpha(ex)\le\alpha(ey)=e\alpha(y).$$

\begin{enumerate} \eroman
    \item[i):]
 If $ex=ey,$ then $e\alpha(x)=e\alpha(y),$ and we have $x+y=ex,$
$\alpha(x)+\alpha(y)=e\alpha(x)=\alpha(ex)$; hence
$\alpha(x)+\alpha(y)=\alpha(x+y).$

\item[ii):]  If $\alpha(x)+\alpha(y)$ is tangible, then certainly
$e\alpha(x)\ne e\alpha(y);$ hence $e\alpha(x)<e\alpha(y).$ This
implies $ex<ey.$ Thus $x+y=y,$ $\alpha(x)+\alpha(y)=\alpha(y);$
hence $\alpha(x)+\alpha(y)=\alpha(x+y).$

\item[iii):] From i) we know that $\alpha(x+y)=\alpha(x)+\alpha(y) $
holds if $ex=ey.$ Assume now that $ex<ey.$ Since $\gamma$ is
injective this implies $e\alpha(x)<e\alpha(y).$ Thus $x+y=y,$
$\alpha(x)+\alpha(y)=\alpha(y);$ hence again
$\alpha(x+y)=\alpha(x)+\alpha(y).$ \end{enumerate}
\end{proof}

 Given an \m-valuation $v: R\to M$, we now focus on the
supervaluations $\varphi: R\to U$ which cover $v,$ i.e., with
$eU=M$ and $e\varphi=\nu_U\circ\varphi=v.$ We single out a class
of supervaluations which will play a special role.

\begin{defn}\label{defn5.11} A supervaluation $\varphi: R\to U$ is
called \bfem{tangibly injective} if the map $\varphi$ is injective
on the set $\varphi^{-1}(\mathcal T(U)),$ i.e.,
$$\forall a,b\in R: \ \varphi(a)=\varphi(b)\in \mathcal T(U)\quad\Rightarrow \quad
a=b.$$
\end{defn}

\begin{example}\label{examp5.12} The supervaluation $\varphi_v:
R\to U(v)$ constructed in \S$4$ (cf.~Example \ref{examp4.5} and
Definition \ref{defn4.6}) is injective on the set $R\setminus
v^{-1}(0),$ hence certainly tangibly injective. Notice that
$\varphi^{-1}(\mathcal T(U(v)))=R\setminus v^{-1}(0),$ i.e.,
$\varphi$ is tangible. $\varphi$ is also surjective.
\end{example}

\begin{thm}\label{thm5.13}
Assume that $\varphi: R\to U$ is a tangibly injective
supervaluation covering $v: R\to M.$ Let $\psi: R\to V$ be another
supervaluation covering $v,$ in particular, $eU=eV=M.$
\begin{enumerate}\item[a)]   $\varphi$ dominates $\psi$ iff the
following holds:
\begin{equation}\label{5.5}
\forall a\in R: \ \varphi(a)=v(a)\quad\Rightarrow \quad
\psi(a)=v(a),
\end{equation}
in other terms, $\varphi(a)\in eU\quad\Rightarrow \quad \psi(a)\in
eV.$ \item[b)] If, in addition, $\varphi$ is tangibly surjective
(cf.~Definition \ref{defn4.1}.c), then $\varphi$ dominates $\psi$
iff there exists a homomorphism map $\alpha: U\to V$ covering the
identity of $M$ such that $\alpha\circ \varphi=\psi.$ The
supervaluation $\psi$ is tangibly surjective iff $\alpha$ is
surjective.
\end{enumerate}
\end{thm}

\begin{proof} a): In the definition of dominance in Definition
\ref{defn5.1}, the axiom D2 holds trivially since
$e\varphi(a)=e\psi(a)=v(a).$ Axiom D3 is our present condition
\eqref{5.5}. Axiom D1 needs only to be checked in the case
$\varphi(a)=\varphi(b)\in \mathcal T(U),$ and then holds trivially
since this implies $a=b$ by the tangible injectivity of $\varphi.$

b): Replacing $U$ by the subsemiring $\mathcal T(U)\cup v(R)$ we
assume without loss of generality that the supervaluation
$\varphi$ is surjective. A transmission $\alpha$ from $\varphi$ to
$\psi$ is forced to cover the identity of $M; $ hence is a
semiring homomorphism, cf.~Proposition \ref{prop5.10}.iii.  We
have $\alpha(U)\supset eV.$ Thus $\alpha$ is surjective iff
$\alpha(\mathcal T(U))=\mathcal T (V).$ This gives us the last
claim.
\end{proof}

\begin{cor}\label{cor5.14}
Assume that $v:R\to M$  is a valuation. The supervaluation
$\varphi_v: R~\to~U(v)$ dominates every supervaluation $\psi: R\to
U$ covering $v.$ Thus these supervaluations $\psi$ correspond
uniquely with the transmissive maps $\alpha: U(v)\to U$ covering
$\id_M.$ They are semiring homomorphisms.
\end{cor}

\begin{proof} $\varphi_v$ is tangibly injective, and \eqref{5.5}
holds trivially, since $\varphi_v(a)\in eU$ only if $v(a)=0.$
Theorem \ref{thm5.13} and \propref{prop5.10}.iii apply.
\end{proof}

\begin{defn}\label{defn5.15}
Due to this property of $\varphi_v$ we call $\varphi_v$ the
\bfem{initial supervaluation} covering $v$ (or \bfem{initial
cover} of  $v$ for short).
\end{defn}

\begin{remark}\label{rem5.16}
We may also regard $v: R\to M$ as a cover of $v,$ viewing $M$ as a
ghost supertropical semiring. Clearly every supervaluation $\psi:
R\to U$ covering $v$ dominates $v$ with transmission $\nu_U.$ Thus
we may view $v: R\to M$ as the  \bfem{terminal supervaluation}
covering~$v$ (or \bfem{terminal cover} of $v$ for short).
\end{remark}

The following proposition gives examples of dominance
$\varphi\ge\psi$ where $\varphi$ is not assumed to be tangibly
injective.

\begin{prop}\label{prop5.17}
Let $U$ be a supertropical semiring with ghost ideal $M:=eU.$
Assume that $L$ is a submonoid of $(M,\cdot)$ with
$M\cdot(M\setminus L)\subset M\setminus L.$ \begin{enumerate}
\item[a)] The map $\alpha: U\to U,$ defined by
$$\alpha(x)=\begin{cases} x &\ \text{if}\quad ex\in L,\\
ex&\ \text{if}\quad ex\in M\setminus L,\end{cases}$$ is an
endomorphism of the semiring $U.$ \item[b)] If $\varphi: R\to U$
is any supervaluation, then the map $\varphi_L:=\alpha\circ
\varphi$ from $R$ to $U$ is a supervaluation dominated by
$\varphi$ and covering the same \m-valuation as $\varphi,$ i.e.
$e\varphi_L=e\varphi.$
\end{enumerate}
\end{prop}

\begin{proof} a):  We have $e\alpha(x)=ex$ for every $x\in U,$ and
$\alpha(x)=x$ for every $x\in M.$ One checks in a straightforward
way that $\alpha$ is multiplicative, $\alpha(0)=0,$ $\alpha(1)=1.$

We verify additivity. Let $x,y\in U$ be given, and assume without
loss of generality that $ex\le ey.$ We have
$e\alpha(x)=\alpha(e)\alpha(x)=\alpha(ex)=ex$ and $e\alpha(y)=ey.$
If $ex=ey$ then $x+y=ex,$ and
$\alpha(x)+\alpha(y)=e\alpha(x)=ex=\alpha(x+y).$ If $ex<ey$ then
$x+y=y$ and $\alpha(x)+\alpha(y)=\alpha(y);$ hence again
$\alpha(x)+\alpha(y)=\alpha(x+y).$

b): Now obvious.
\end{proof}

Notice that $\varphi_L=\alpha\circ\varphi$ with a map $\alpha:
U\to U$ given by $\alpha(x)=x$ if $ex\in L,$ and $\alpha(x)=ex$ if
$ex\in M\setminus L.$ Thus  if $\varphi$ is surjective, $\alpha$
is the transmission from $\varphi$ to $\varphi_L$.

It is not difficult to find instances where \propref{prop5.17}
applies.

\begin{example}\label{examp5.18}
  Assume  that $M$
is a submonoid of $\Gamma\cup\{0\}$ for $\Gamma$ an ordered
abelian group. Let $H$ be a  subgroup of $\Gamma$ containing the
set $\{x\in M\bigm | x>1\}.$ Then
$$L=\{x\in M\mid \exists h\in H\quad\text{with}\quad x\ge h\} $$
is a  submonoid of $M\setminus\{0\}.$ We claim that $M \cdot
(M\setminus L)\subset M\setminus L.$
\end{example}

\begin{proof}
Let $x\in M,$\ $y\in M\setminus L$ be given. If $x\le 1,$ then
$xy\le y;$ hence, clearly, $xy\in M\setminus L.$ Assume now that
$x>1.$ Then $x\in H.$ Suppose that $xy\in L;$ hence $h\le xy$ for
some $h\in H.$ Then $x^{-1}\le y$ and $x^{-1}h\in H;$ hence $y\in
L,$ a contradiction. Thus $xy\in M\setminus L$ again.
\end{proof}

In \cite{IKR} we will meet many transmissions which are not
semiring homomorphisms.

\section{Fiber contractions}\label{sec:6}

Before we come to the main theme of this section, we write down
functorial  properties of the class of  transmissive maps.

\begin{prop}\label{prop6.1}
Let $\alpha: U\to V$ and $\beta:V\to W$ be maps between
supertropical semirings.
\begin{enumerate}
\item[i)] If $\alpha$ and $\beta$ are  transmissive, then
$\beta\alpha$ is   transmissive. \item[ii)] If $\alpha$ and
$\beta\alpha$ are   transmissive and $\alpha$ is surjective, then
$\beta$ is   transmissive.
\end{enumerate}

\end{prop}

\begin{proof}
a) It is evident that analogous statements hold for the class of
maps between supertropical semirings obeying the axioms TM1--TM4
in \S5. Thus we may assume from the beginning that $\alpha,\beta$
and (hence) $\beta\alpha$ obey TM1--TM4, and have only to deal
with the axiom TM5   (cf. \thmref{thm5.4}, Definition
\ref{defn5.5a}).

b) We conclude from TM3 and TM4 that $\alpha$ maps $eU$ to $eV$
and $\beta$ maps $eV$ to $eW.$ TM5 demands that these restricted
maps are semiring homomorphisms. Thus it is evident that
$\beta\alpha$ obeys TM5 if $\alpha$ and $\beta$ do. If $\alpha$ is
surjective, then also the restriction $\alpha | eU: eU\to eV$ is
surjective, since for $x\in U,$ \ $y\in eV$ with $\alpha(x)=y$ we
also have $\alpha(ex)=y.$ Clearly, TM5 for $\alpha$ and
$\beta\alpha$ implies TM5 for $\beta$ in this case.
\end{proof}

Often we will only need the following special case of
\propref{prop6.1}.

\begin{cor}\label{cor6.2}
Let $U,V,W$ be supertropical semirings. Assume that $\alpha: U\to
V$ is a surjective semiring homomorphism. Then a map $\beta: V\to
W$ is transmissive  iff $\beta\alpha$ has this
property.\hfill\quad\qed
\end{cor}

In the entire section $U$ is a \textit{supertropical semiring}. We
look for equivalence relations on the set $U$ that respect the
multiplication on $U$ and the fibers of the ghost map $\gamma_U:
U\to eU.$

\begin{defn}\label{defn6.3} Let $E$ be an equivalence relation on
the set $U$. We say that $E$ is \bfem{multiplicative} if for any
$x_1,x_2,y\in U,$
\begin{equation}\label{6.3}
x_1\sim_E x_2\quad\Rightarrow\quad x_1y\sim_E x_2 y.\end{equation}
We say that $E$ is \bfem{fiber conserving} if for any $x_1,x_2\in
U,$
\begin{equation}\label{6.4}
x_1\sim_E x_2\quad\Rightarrow\quad ex_1=ex_2.\end{equation} If $E$
is both multiplicative and fiber conserving, we call $E$ an
\bfem{MFCE-relation} (multiplicative fiber conserving equivalence
relation) for short.
\end{defn}

\begin{examples}\label{examps6.4} $ $
\begin{enumerate} \eroman
    \item
 Assume that $\alpha: U\to V$ is a multiplicative map from $U$
to a supertropical semiring $V.$ Then the equivalence $E(\alpha)$,
given by
$$x_1\sim x_2\quad\text{iff}\quad \alpha(x_1)=\alpha(x_2),$$
is clearly multiplicative. If in addition $\alpha(e_U)=e_V,$ and
 if the induced map $\gamma: eU\to eV,$ \
$\gamma(ex)=e\alpha(x),$ is injective, then $E(\alpha)$ is also
fiber conserving; hence an MFCE-relation. We usually denote this
equivalence $\sim$ by $\sim_\alpha.$

In particular, we have an MFCE-relation $E(\alpha)$ on $U$ for any
semiring homomorphism $\alpha: U\to V$ which is injective on $eU.$

\item  The ghost map $\nu=\nu_U: U\to U$ gives us an MFCE-relation
$E(\nu)$ on $U.$ Clearly
$$x_1\sim_\nu x_2\quad\text{iff}\quad ex_1=ex_2.$$
$E(\nu)$ is the coarsest MFCE-relation on $U.$

\item  If $E_1$ and $E_2$ are equivalence relations on the set $U$,
then $E_1\cap E_2$ is again an equivalence relation on $U.$ $\{$As
usual, we regard an equivalence relation on $U$ as a subset of
$U\times U \}.$ We have
$$x_1\sim_{E_1\cap E_2} x_2\quad\text{iff}\quad
x_1\sim_{E_1}x_2\quad \text{and}\quad x_1\sim_{E_2}x_2.$$ If $E_1$
is multiplicative and $E_2$ is  an MFCE, then $E_1\cap E_2$ is an
MFCE.

\item In particular, every multiplicative equivalence relation $E$
on $U$ gives us an MFCE-relation $E\cap E(\nu)$ on $U.$ This is
the coarsest MFCE-relation on $U$ which is finer than $E.$ We have
$$x_1\sim_{E\cap E(\nu)}x_2\quad\text{iff}\quad x_1\sim_E
x_2\quad\text{and}\quad ex_1=ex_2.$$

\item We define an equivalence relation $\tE$ (the ``$t$'' alludes
to ``tangible'') on $U$ as follows, writing $\sim_t$ for
$\sim_{\tE}:$ \begin{align*} x_1\sim_t x_2\quad \text{iff either}&
\quad
x_1=x_2\quad\\
\text{or} & \quad   x_1,x_2\in \mathcal T(U)\quad \text{and}\quad
ex_1=ex_2.\end{align*} Clearly, this is an MFCE-relation iff for
any tangible $x_1,x_2,y\in E$ with $ex_1=ex_2$ both $x_1y$ and
$x_2y$ are tangible or equal. In particular, $\tE$ is an MFCE if
$\mathcal T(U)$ is closed under multiplication. \end{enumerate}
\end{examples}

Let $F$ denote the equivalence relation on $U$ which has the
equivalence classes $\mathcal T(U)$ and $eU$. It is readily
checked that $\tE=F\cap E(\nu).$

\medskip

The equivalence classes of $\tE$ contained in $\mathcal T(U)$ are
the sets $\mathcal T(U)\cap\nu_U^{-1}(z)$ with $z\in M,$ which are
not empty. We call them the \bfem{tangible fibers} of $\nu_U.$

\medskip

Our next goal is to prove that, given an MFCE-relation $E$ on $U,$
the set $U/E$ of all $E$-equivalence classes inherits from $U$ the
structure of a supertropical semiring.

\begin{lem}\label{lem6.5} If $E$ is a fiber conserving equivalence
relation on $U$, then for any $x_1,x_2,y\in U$ $$x_1\sim_E
x_2\quad\Rightarrow\quad x_1+y\sim_E x_2+y.$$
\end{lem}

\begin{proof} $ex_1=ex_2.$ If $ey<e x_1,$ we have $x_1+y=x_1,$\
$x_2+y=x_2.$ If $ey=ex_1,$ we have $x_1+y=ey=x_2+y .$ If
$ey>ex_1,$ we have $x_1+y=y=x_2+y.$ Thus, in all three cases,
$x_1+y\sim_E x_2+y.$
\end{proof}
Notice that, as a formal consequence of the lemma, more generally
$$ x_1 \sim_E x_2, \   y _1 \sim_E y_2  \quad \Rightarrow  \quad  x_1 + y_1 \sim_E x_2 + y_2. $$

\begin{thm}\label{thm6.6}
Let $E$ be an MFCE-relation on a supertropical semiring $U.$ On
the set $\overline U: = U/E$ of equivalence classes $[x]_E,$ $x\in
U,$ we have a unique semiring structure such that the projection
map $\pi_E: U\to \overline U,$\ $x\mapsto [x]_E$ is a semiring
homomorphism. This semiring $\overline U$ is supertropical, and
$\pi_E$ covers a semiring isomorphism $eU\osr \bar e\overline U.$
(Here $\bar e:=e_{\overline U}=\pi_E(e).$)
\end{thm}

\begin{proof}
We write $\bar x: =[x]_E$ for $x\in U$ and $\pi:=\pi_E.$ Thus
$\pi(x)=\bar x.$ Due to \lemref{lem6.5} and condition \eqref{6.3},
we have a well-defined addition and multiplication on $\overline
U,$ given by the rules $(x,y\in U)$
$$\bar x+\bar y:=\overline{x+y},\qquad \bar x\cdot \bar
y:=\overline{xy}.$$

The axioms of a commutative semiring are valid for these
operations, since they hold in $U,$ and the map $\pi$ is a
homomorphism from $U$ onto the semiring $\overline U$.

We have $\bar 1+ \bar 1 = \bar e$ and $\bar e \overline U = \pi
(eU)$. If $x,y \in eU$ and $x \sim_E y$  then $x =ex = ey =y$,
since $E$ is fiber conserving. Thus the restriction $\pi | eU$ is
an isomorphism from the bipotent semiring $eU$  onto the semiring
$\bar e \overline U$ (which thus is again bipotent).

We are ready to prove that $\bU$ is supertropical, i.e. that
axioms $(3.3')$, $(3.3'')$, $(3.4)$ from \S\ref{sec:3} are valid.
It is obvious that $\bU$ inherit  properties $(3.3')$ and $(3.4)$
from $U$. Let $x,y \in E$ be given with $\be \bx = \be \by$, i.e.
$\overline{ex} = \overline{ey}$. Then $ex = ey$; hence $x + y =
ex$ by axiom $(3.3'')$ for $U$. Applying the homomorphism $\pi$ we
obtain $\bx + \by = \be \bx$. Thus $\bU$ also obeys $(3.3'')$.
\end{proof}

\begin{remark}\label{rem6.7}
\thmref{thm6.6} tells us, in particular, that every MFCE-relation
$E$ on $U$ is of the form $E(\alpha)$ for some semiring
homomorphism $\alpha: U\to V$ with $\alpha| eU$ bijective, namely,
$E=E(\pi_E).$\end{remark}

\begin{thm}\label{thm6.8}
Assume that $\alpha: U\to V$ is a multiplicative map. Let $E$ be
an MFCE-relation on $U,$ which is respected by $\alpha,$ i.e.,
$x\sim_E y$ implies $\alpha(x)=\alpha(y).$ Clearly, we have a
unique multiplicative map $\bar\alpha: U/E\to V$ with
$\bar\alpha\circ\pi_E=\alpha.$

 Then, if $\alpha$ is a
transmission (a semiring homomorphism), the map $\bar\alpha$ is of
the same kind.
\end{thm}

\begin{proof}
Corollary \ref{cor6.2} gives us all the claims, since $\pi_E$ is a
surjective homomorphism.
\end{proof}

\begin{defn}\label{defn6.10}
We call a map $\alpha:U\to V$ between supertropical semirings a
\bfem{fiber contraction}, if $\alpha$ is   transmissive and
surjective, and the map $\gamma: eU\to eV$ covered by $\alpha$ is
strictly order preserving.
\end{defn}

Notice that then $\alpha$ is a semiring homomorphism
(cf.~\propref{prop5.10}.iii) (hence $\alpha$ is a transmission),
and $\gamma$ is an isomorphism from $eU$ to   $eV.$

\begin{schol}\label{schol6.11}$ $ \begin{enumerate}\item[i)]  If $E$ is an
MFCE-relation on $U$, by \thmref{thm6.6}, the map $\pi_E: U\to
U/E$ is a  fiber contraction. On the other hand, if a surjective
fiber contraction $\alpha: U\twoheadrightarrow V$ is given, then
clearly $E(\alpha)$ is an MFCE-relation, and, as \thmref{thm6.8}
tells us, $\alpha$ induces a semiring isomorphism $\bar{\alpha}:
U/E(\alpha)\osr V$ with $\alpha=\bar\alpha\circ\pi_{E(\alpha)}.$
In short, every   fiber contraction $\alpha$ on $U$ is a map
$\pi_E$ with $E$ an MFCE-relation on $U$ uniquely determined by
$\alpha,$ followed by a semiring isomorphism. \item[ii)] If the
semiring isomorphism $\bar\alpha$ is the identity $\id_M$ of
$M:=eU$ (in particular $eU=eV$), we say $\alpha$ is a \bfem{fiber
contraction over} $M$. \end{enumerate}
\end{schol}

If $E$ is an equivalence relation on a set $X$, and $Y$  is a
subset of $X$, we denote the set of all equivalence classes
$[x]_E$ with $ x\in Y\} $ by $Y/E.$

\begin{examp}\label{examp6.12} Assume that $U$ is a supertropical
domain (cf. \ref{defn3.13}). Then the equivalence relation $\tE$
introduced in Example \ref{examps6.4}.v is MFCE, and $\mathcal
T(U)$ is a union of $\tE$-equivalence classes.  The ring
$\overline U=U/\tE$ is a supertropical domain with $\mathcal
T(\overline U)=\mathcal T(U)/\tE$ and $\mathcal G(\overline
U)=\mathcal G(U).$ The ghost map of $\overline U$ maps $\mathcal
T(\overline U)$ bijectively to $\mathcal G(U);$ hence gives us a
monoid isomorphism $v: \mathcal T(\overline U)\osr \mathcal G(U).$
Thus (in notation of Examples \ref{examps3.16})
$$ U/\tE=D(\mathcal G(U)).$$
The map $\pi_{\tE}$ is a fiber contraction over
$eU=eU/\tE.$\end{examp}

\begin{examp}\label{examp6.13} (cf.~\propref{prop5.17}) Let $U$ be
a supertropical semiring, $M:=eU,$ and let $L$ be a submonoid of
$(M,\cdot)$ with $M\cdot (M\setminus L)\subset M\setminus L.$ Then
the map $\alpha: U\to U$ with $\alpha(x)=x$ if $ex\in L$,
$\alpha(x)=ex$ if $ex\in M\setminus L$, is a fiber contraction
over $M.$ The image of $\alpha$ is the subsemiring
$\nu_U^{-1}(L)\cup(M\setminus L)$ of $U.$\end{examp}

\begin{examp}\label{examp6.14}
Let again $U$ be a supertropical semiring and $M:=eU.$ But now
assume only that $L$ is a \bfem{subset} of $M$ with
$M\cdot(M\setminus L)\subset M\setminus L.$ We define an
equivalence relation $E(L)$ on $U$ as follows:
$$x\sim_{E(L)}y\quad \Leftrightarrow \quad\text{either}\quad x=y
\quad\text{or}\quad ex=ey\in M\setminus L.$$ One checks easily
that $E(L)$ is MFCE. But if $L$ is not a submonoid of $(M,\cdot)$,
then in the supertropical semiring $\overline U:=U/E(L)$ the set
$\mathcal T(\overline U)$ of tangible elements is not closed under
multiplication. In particular, $\overline U$ is not isomorphic to
a subsemiring of $U.$\end{examp}

For later use we introduce one more notation.

\begin{notation}\label{not6.14} If $\vrp:R \to U$ is a
supervaluation and $E$ is an \MFCE-relation on $U$, let $\vrp / E$
denote the supervaluation  $\pi_E \circ \vrp : R \to U/ E$. Thus,
for any $a \in R$,
$$ (\vrp/E)(a) \ := \ [\vrp(a)]_E \ . $$
\end{notation}

\section{The lattices $C(\varphi)$ and $\Cov(v)$}\label{sec:7}

Given an \m-valuation $v: R\to M$ on a semiring $R,$ we now can
say more about the class of all supervaluations $\varphi$ covering
$v.$ Recall that these are the supervaluations $\varphi: R\to U$
with $eU=M$ and $\nu_U\circ\varphi=v,$ in other words,
$e\varphi=v.$ For short, we call these supervaluations $\varphi$
the \bfem{covers} of the \m-valuation $v.$ It suffices to focus on
covers of $v$ which are tangibly surjective, cf.~Remark
\ref{rem4.4}. (N.B. Without loss of generality, we could even
assume that $v$ is surjective. Then a cover $\varphi$ of $v$ is
tangibly surjective iff $\varphi$ is surjective.)

\begin{defn}\label{defn7.1}\end{defn}

 \begin{enumerate}
 \item[a)] We call two covers $\varphi_1:R\to U_1,$\ $\varphi_2:
 R\to U_2$ of $ v$
\bfem{equivalent}, if $\varphi_1\ge \varphi_2$ and
$\varphi_2\ge~\varphi_1,$ i.e., $\varphi_1$ dominates $\varphi_2$,
and $\varphi_2$ dominates $\varphi_1.$ If $\varphi_1$ and
$\varphi_2$ are tangibly surjective (without essential loss of
generality, cf.~Remark \ref{rem4.4}), this means that
$\varphi_2=\alpha\circ\varphi_1$ with $\alpha: U_1\to U_2$ a
semiring isomorphism over $M$ (i.e., $e\alpha(x)=ex$ for all $x\in
U_1). $ \item[b)] We denote the equivalence class of a cover
$\varphi: R\to U$ of $v$ by $[\varphi]$, and we denote the set of
all these equivalence classes by $\Cov(v)$.
 $\{$Notice that
$\Cov(v)$ is really a set, not just a class, since for any
tangibly surjective cover $\varphi: R\to U$, we have
$U=\varphi(R)\cup M$; hence the cardinality of $U$ is bounded by
$\Card R+\Card M.\}$
 On $\Cov(v)$ we have a partial
ordering: $[\varphi_1]\ge[\varphi_2]$ iff $\varphi_1$ dominates
$\varphi_2.$ We always regard $\Cov(v)$ as a poset\footnote{\ =
partially ordered set} in this way. \item[c)] If a covering
$\varphi: R\to U$ of $v$ is given, we denote the subposet of
$\Cov(v)$ consisting of all $[\psi]\in\Cov(v)$ with
$[\varphi]\ge[\psi]$ by $C(\varphi).$ $\{$Notice that this poset
is determined by $\varphi$ alone, since $v=e\varphi.\}$
\end{enumerate}


\medskip

In \S5 we have seen that, given a tangibly surjective cover
$\varphi: R\to U$ of $v,$ the tangibly surjective covers $\psi:
R\to V$ dominated by $\varphi$ correspond uniquely to the
transmissive surjective maps $\alpha: U\to V$ which restrict to
the identity on $M=eU=eV.$ Scholium \ref{schol6.11} from the
preceding section tells us,  in particular, the following.

\begin{thm}\label{thm7.2} Assume that $\varphi: R\to U$ is a tangibly
surjective covering of the \m-valuation $v: R\to M.$
\begin{enumerate}\item[a)] The elements $[\psi]$ of $C(\varphi)$
correspond uniquely to the MFCE-relations $E$ on $U$ via
$[\psi]=[\varphi / E]$.

\item[b)] Let $\MFC(U)$ denote the set of
all \MFCE-relations on $U,$ ordered by the coarsening relation:
$E_1\le E_2$ iff $E_2$ is coarser than $E_1,$ i.e., $E_1\subset
E_2$, if the $E_i$ are viewed -- as customary -- as subsets of
$U\times U.$ The map $E\mapsto [\varphi / E]$ is an
anti-isomorphism (i.e., an order reversing bijection) from the
poset $\MFC(U)$ to the poset $C(\varphi).$\end{enumerate}
\end{thm}

If $(E_i\bigm| i\in I)$ is a family in $\MFC(U)$ then the
intersection $E:=\bigcap_{i\in I}E_i$ is again an MFCE-relation on
$U,$ and is the infimum of the family $(E_i\bigm|i\in I)$ in
$\MFC(U)$. Since $\MFC(U)$ has a biggest and smallest element,
namely $E(\nu_U)$  and the diagonal of $U$ in $U\times U,$ it is
now clear that the poset $\MFC(U)$ is a complete lattice. Thus,
for any cover $\varphi: R\to U$ of the \m-valuation $v: R\to M$,
also the poset $C(\varphi)$ is a complete lattice. $\{$We easily
retreat to the case that $\varphi$ is tangibly surjective.$\}$

The supremum of a family $(E_i\bigm|i\in I)$ in $\MFC(U)$ is the
following equivalence relation $F$ on $U.$ Two elements $x,y$ of
$U$ are $F$-equivalent iff there exists a finite sequence
$x_0=x,x_1,\dots,x_m=y$ in $U$ such that for each
$j\in\{1,\dots,m\}$ the element $x_{j-1}$ is $E_k$-equivalent to
$x_j$ for some $k\in I.$

\begin{construction}\label{constr7.3}
Assume again that $\varphi$ is tangibly surjective. The supremum
$\bigvee_{i\in I}\xi_i$ of a family $(\xi_i\bigm| i\in I)$ in
$C(\varphi)$ can be described as follows. Choose for each $i\in I$
a tangibly surjective representative $\psi_i: R\to V_i$ of
$\xi_i.$ Thus $eV_i=M,$ and $\psi_i$ is a cover of $v$ dominated
by $\varphi.$ Let $e_i:=e_{V_i}$ $(=1_M)$, and let $V$ denote the
set of all elements $x=(x_i\bigm| i\in I)$ in the semiring
$\prod_{i\in I}V_i$ with $e_ix_i=e_jx_j$ for $i\ne j.$ This is a
subsemiring of $\prod_{i\in I} V_i$ containing the image $M'$ of
$M$ in $\prod V_i$ under the diagonal embedding of $M$ into $\prod
V_i.$ We identify $M'=M,$ and then have
$$e_U=1_M=(e_i\bigm| i\in I)=1_V+1_V.$$ It is now a trivial matter
to verify that $V$ is a supertropical semiring by checking the
axioms in \S3. We have $e_VV=eV=M'=M.$ The supervaluations
$\psi_i:R\to U_i$ combine to a map $\psi: R\to V,$ given by
$$\psi(a):=(\psi_i(a)\bigm|i\in I)\in V$$ for $a\in R.$ It is a
supervaluation covering $v,$ and $\varphi: R\to U$ dominates
$\psi$ (e.g., check the axioms D1--D3 in \S$5$). The class
$[\psi]$ is the supremum of the family $(\xi_i\bigm|i\in I)$ in
$C(\varphi).$
\end{construction}

Given again a family $(\xi_i\bigm|i\in I)$ in $C(\varphi)$ with
representatives $\psi_i: R\to V_i$ of the $\xi_i,$ we indicate how
the infimum $ \bigwedge \xi_i$ in $C(\varphi)$ can be built,
without being as detailed as above for the supremum.

We assume that each supervaluation $\psi_i$ is surjective. The
transmission $\delta_i: U\to V_i$ from $\varphi$ to $\psi_i$ is a
surjective semiring homomorphism. We form the categorical direct
limit (= colimit) of the family $(\delta_i\bigm|i\in I)$ in the
category of semirings (cf. \cite[Chap.~II]{Mit}, \cite[III,
\S3]{ML}). Thus we have a semiring $V$ together with a family of
semiring homomorphisms $(\alpha_i: V_i\to V\bigm|i\in I)$ such
that $\alpha_i\circ\delta_i=\alpha_j\circ\delta_j$ for $i\ne j$,
which is universal. This means that, given a family $(\beta_i:
V_i\to W\bigm|i\in I)$ of homomorphisms with
$\beta_i\circ\delta_i=\beta_j\circ\delta_j$ for $i\ne j, $ there
exists a unique homomorphism $\beta: V\to W$ with $\beta\circ
\alpha_i=\beta_i$ for every $i\in I.$ Choosing some $i\in I$ let
$$\varepsilon:=\alpha_i\circ\delta_i:U\to V.$$
This homomorphism, which is independent of the choice of $i,$ is
surjective, due to universality, since all maps $\delta_j: U\to
V_j$ are surjective.  It turns out that the restriction
$\varepsilon|eU$ maps $eU=M$ isomorphically onto $eV.$ We identify
$M$ with $eV$ by this isomorphism and then have $\varepsilon
|eU=1_M.$

  This can be  seen as follows. Let
$\nu:=\nu_U$ and $\nu_i:=\nu_{V_i}$ denote the ghost maps of $U$
and $V_i.$ For every $i\in I$ we have $\nu_i\circ\delta_i=\nu.$ By
universality we obtain a homomorphism $\mu:  V\to M$ with
$\mu\circ\alpha_i=\nu_i$ for every $i.$ Let $j_i$ denote the
inclusion map from $M$ to $V_i.$ We have $\nu_i\circ j_i=\id_M;$
hence
$$\mu\circ\alpha_i\circ j_i=\nu_i\circ j_i=\id_M.$$
The surjective homomorphism $\alpha_i$ maps $M=eV_i$ onto $eV.$ We
conclude that the restriction $\alpha_i | M$ gives an isomorphism
from $M$ onto $eV,$ the inverse map being given by $\mu.$

We identify $M$ with $eV$ via $\alpha_i | M.$ Now $\alpha_i:
V_i\to V$ has become a surjective semiring homomorphism over $M$
(for every $i).$ Thus also $\varepsilon: U\to V$ is a surjective
homomorphism over $M.$ We conclude, that $\ep$ gives an
\MFCE-relation $E(\ep)$ and the semiring  $V$ is supertropical.
The supervaluation $$\psi:=\varepsilon \varphi=\alpha_i\circ
\psi_i \quad \text{is dominated by every}\quad
\psi_i\quad\text{and}\quad [\psi]=\bigwedge_{i}\xi_i.$$ Since
$V_i=\psi_i(R)\cup M$ for every $i,$ the semiring $V$ and the
$\alpha_i$ can be described completely in terms of the $\psi_i$
without mentioning $U$ and the $\delta_i.$ We leave this to the
interested reader.

\begin{defn}\label{defn7.4}
We call a supervaluation $\varphi$ \bfem{initial} if $\varphi$
dominates every other supervaluation $\psi$ with $e\varphi=e\psi.$
We then also say that $\varphi$ is an \bfem{initial cover} of
$v:=e\varphi.$\end{defn}

If an \m-valuation $v: R\to M$ is given, a supervaluation
$\varphi: R\to U$ is an initial cover of $v$ iff $e\varphi=v$ and
$[\varphi]$ is the biggest element of the poset $\Cov(v).$

Such an initial cover had been constructed explicitly in \S$4$ in
the case that $v$ is a   valuation, namely, the supervaluation
$\varphi_v: R\to U(v),$ cf.~Definition \ref{defn4.6} and Corollary
\ref{cor5.14}. We now prove that an initial cover always exists,
although in general we do not have an explicit description.

\begin{prop}\label{prop7.5} Every \m-valuation $v: R\to M$ has an
initial cover. The poset $\Cov(v)$ is a complete
lattice.\end{prop}

\begin{proof} Let $(\psi_i\bigm| i\in I)$ be a family of coverings
of $v$ which represents every element of the set $\Cov(v).$ Now
repeat Construction \ref{constr7.3} with this family. It gives us
a covering $\psi: R\to V$ of $v$ which dominates all $\psi_i;$
hence is an initial covering of $v.$ Of course, $C(\psi)=\Cov(v),$
and thus $\Cov(v)$ is a complete lattice.\end{proof}

\begin{notation}\label{not7.6} If $v: R\to M$ is any
\m-valuation, let $\varphi_v: R\to U(v)$, denote a fixed tangibly
surjective initial supervaluation covering $v.$ If $v$ is a
valuation, we choose for $\varphi_v$ the supervaluation
constructed in Example  \ref{examp4.5}.\end{notation}

Notice that $\varphi_v$ is unique up to unique isomorphism over
$M,$ i.e., if $\psi: R\to V$ is another surjective initial cover
of  $v$, there exists a unique semiring isomorphism $\alpha:
U(v)\osr V$ which restricts to the identity on $M.$ We call
$\varphi_v$ ``\bfem{the}'' \bfem{initial cover} of $v.$ The
lattice $\Cov(v)$ coincides with $C(\varphi_v).$

Given a supervaluation $\varphi: R\to U$ or an \m-valuation $v:
R\to M$, we view the lattice $C(\varphi)$ and $\Cov(v)$ as a
measure of complexity of $\varphi$ and $v$, respectively, and thus
make the following formal definition.

\begin{defn}\label{defn7.7} We call the isomorphism class of the
lattice $C(\varphi)$ the \bfem{lattice complexity} of the
supervaluation $\varphi$ and denote it by $\lc(\varphi).$ In the
same vein we call the isomorphism class of the lattice $\Cov(v)$
the \bfem{tropical  complexity} of the \m-valuation $v$ and denote
it by $\trc(v).$ We have $\trc(v)=\lc(\varphi_v).$\end{defn}

The word ``complexity" in Definition \ref{defn7.7} should not be
taken too seriously. Usually a ``measure of complexity" has values
in natural numbers or, more generally, in some well understood
fixed ordered set. The isomorphism classes of lattices are not
values of this kind. Our idea behind the definition  is that, if
you are given a function $m$ on the class of lattices which
measures (part of) their complexity in some way, then $m\circ
\lc,$ resp. $m\circ \trc,$ is such a function on the class of
supervaluations, resp. \m-valuations.

\begin{thm}\label{thm7.8} If $\varphi: R\to U$ and $\varphi':
R'\to U$ are tangibly surjective supervaluations with values in
the same supertropical semiring $U$, then
$\lc(\varphi)=\lc(\varphi').$\end{thm}

\begin{proof} Both lattices $C(\varphi)$ and $C(\varphi')$ are
anti-isomorphic to $\MFC(U);$ hence are isomorphic.\end{proof}

This result is quite remarkable, since it says that the lattice
complexity of a surjective supervaluation $\vrp: R \to U$ depends
only on the isomorphism class of the target semiring $U$.

\begin{examp}\label{examp7.9} Let $\varphi: R\to U$ be a tangibly
surjective supervaluation. The identity \\ $\id_U: U\to U$ is also
a supervaluation. It is the initial cover of the ghost map $\nu_U:
U\to eU.$ We have $\lc(\varphi)=\trc(\nu_U).$\end{examp}

\section{Orbital equivalence relations}\label{sec:8}

Our main goal in this section is to introduce and study a special
kind of MFCE-relations on supertropical semirings, which seems to
be more accessible than MFCE-relations in general. But for use in
later sections, we will define  more generally ``orbital"
equivalence relations on supertropical semirings. They are
multiplicative but not necessarily fiber conserving. The relations
we are looking  for here then will be the orbital \MFCE-relations.

In the following  $U$ is a supertropical semiring, and $M:=eU$
denotes its ghost ideal. We always assume that $\mathcal T(U)$ is
not empty, i.e., $e\ne1.$ We introduce the set
$$S(U):=\{x\in U\bigm| x\mathcal T(U)\subset \mathcal T(U)\}.$$
This is a subset of $\mathcal T(U)$ closed under multiplication
and containing the unit element $1_U;$ hence is a monoid.

The monoid $S(U)$ operates on the sets $U$ and $\mathcal T(U)$ by
multiplication. If $\mathcal T(U)$ itself is closed under
multiplication then $S(U)=\mathcal T(U).$

Let $G$ be a submonoid of $S(U).$ Then also $G$ operates on $U$
and on $\mathcal T(U).$ For any $x\in U$ we call the set $Gx$ the
\bfem{orbit} of $x$ under $G$ (as common at least for $G$ a
group). We define a binary relation $\sim_G$ on $U$ as follows:
$$x\sim_Gy\quad\Leftrightarrow\quad \exists g,h\in G: gx=hy.$$
Thus $x\sim_G y$ iff the orbits $Gx$ and $Gy$ intersect. Clearly
this is an equivalence relation on $U,$ which is multiplicative,
i.e., obeys the rule \eqref{6.3} from \S$6$. We denote this
equivalence relation by $E(G).$

The relation $E(G)$ on $U$ is MFCE, i.e., obeys also the rule
\eqref{6.4} from \S6, iff $G$ is contained in the ``unit-fiber''
$$\mathcal T_e(U):=\{x\in \mathcal T(U)|ex=e\}$$ of $\mathcal T(U).$ The biggest such monoid
is the unit fiber
$$S_e(U):=\{g\in S(U)\bigm| eg=e\}=\mathcal T_e(U)\cap S(U)$$
of $S(U)$.

\begin{examp}\label{examp8.1} Assume that $R$ is a field and $v:
R\to \Gamma\cup\{0\}$ is  a surjective valuation on $R.$ $\{$In
classical terms, $v$ is a Krull valuation on $R$ with value group
$\Gamma.\}$ Let
$$U:=U(v)=(R\setminus\{0\})\ \dot\cup\ \Gamma\ \dot\cup\ \{0\},$$
cf.~Definition \ref{defn4.6}. Then $S(U)$ is the multiplicative
group $R^*=R\setminus\{0\}$ of the field $R,$ and $S_e(U)$ is the
group $\mathfrak o_v^*$ of units of the valuation domain
$$\mathfrak o_v:=\{x\in R\bigm| v(x)\le 1\}.$$
\end{examp}

\begin{defn}\label{defn8.2} We call an equivalence relation  $E$ on
the supertropical semiring $U$ \bfem{orbital} if $E=E(G)$ for some
submonoid $G$ of $S(U).$ We denote the set of all orbital
equivalence relations on $U$ by $\Orb(U)$ and the subset
$\Orb(U)\cap \MFC(U),$ consisting of the orbital MFCE-relations on
$U,$ by $\OFC(U).$ $\{$``OFC" alludes to ``orbital fiber
conserving".$\}$ Consequently, we call the elements of $\OFC(U)$
the \bfem{orbital fiber conserving equivalence relations} on $U$,
or \bfem{OFCE-relations} for short.\end{defn}

\begin{examp}\label{examp8.3}
It is evident that $E(S(U))$ is the coarsest orbital equivalence
relation and $F:=E(S_e(U))$ is the coarsest OFCE-relation on $U.$
Assume now that $U$ is a supertropical domain. Then $S(U)=\mathcal
T(U),$\ $S_e(U)=\mathcal T_e(U),$ and $\mathcal G(U)=e\mathcal
T(U).$ \ $E(S(U))$ has just $3$ equivalence classes, namely,
$\mathcal T(U),$ $\mathcal G(U)$ and $\{0\}.$ On the other hand,
$F$ is finer than the MFCE-relation $\tE$ introduced in Example
\ref{examps6.4}.v, whose equivalence classes in $\mathcal T(U)$
are the tangible fibers of the ghost map $\nu_U.$ Very often $\tE$
is not orbital; hence $F\subsetneqq \tE.$
\end{examp}

\begin{subexamp}\label{subexamp8.4}
Let $R=k[x]$ be the polynomial ring in one variable $x$ over a
field $k.$ Choose a real number $\vartheta$ with $0<\vartheta<1,$
and let $v$ be the surjective valuation on $R$ defined by
$$v(f)=\vartheta^{\deg f}.$$
Thus, $v: R\twoheadrightarrow G\cup\{0\}$ with $G$ the monoid
$\{\vartheta^n\bigm| n\in\mathbb N_0\}\subset\mathbb R.$ Finally,
take
$$U:=U(v)=(R\setminus\{0\})\cup G\cup\{0\},$$
cf.~Definition \ref{defn4.6}. We have $S(U)=R\setminus\{0\}$ and
$$S_e(U)=\{f\in R\bigm| \deg f=0\}=k\setminus\{0\},$$
the set of nonzero constant polynomials. If $f,g\in \mathcal T(U)$
are given with $ef=eg,$ i.e., $\deg f=\deg g,$ then $f\sim_F g$
iff $g=cf$ with $c$ a constant $\ne0.$ Thus, the set of
$F$-equivalence classes in $\mathcal T(U)$ can be identified with
the set of monic polynomials in $k[x],$ while the
$\tE$-equivalence classes are the sets $\{f\in k[x]\bigm|\deg
f=n\}$ with $n$ running  thorough $\mathbb N_0.$ For $n=0$ this
$\tE$-equivalence class is also an $F$-equivalence class, while
for $n>0$ it decomposes into infinitely many $F$-equivalence
classes if the field $k$ is infinite, and into $|k|^n$
$F$-equivalence classes if $k$ is finite.

The semiring $U/F$ (cf. \S6) can be identified with the
subsemiring $V$ of $U$, which has as tangible elements the monic
polynomials in $k[x]$ and has the same ghost ideal $eV=eU$ as~$U.$
\hfill\quad \qed
\end{subexamp}

Different submonoids $G,H$ of $S(U)$ may yield the same orbital
equivalence relation $E(G)=E(H).$ But this ambiguity can be tamed.

\begin{prop}\label{prop8.5} If $G$ is a submonoid of $S(U)$, then
$$G':=\{x\in S(U)\bigm| \exists g\in G: gx\in G\}$$ is a submonoid
of $S(U)$ containing $G,$ and $E(G)=E(G').$ If $G\subset S_e(U)$
then $G'\subset S_e(U).$\end{prop}

\begin{proof}
a) It is immediate that $G'$ is a submonoid of $S(U)$ and that
$G\subset G'.$ Given $x\in G'$ we have elements $g,h\in G$ with
$gx=h.$ If in addition $G\subset S_e(U)$, then $e=eh=(eg)(ex)=ex;$
hence $x\in S_e(U).$ Thus $G'\subset S_e(U).$ It follows from
$G\subset G'$ that $E(G)\subset E(G').$

b) Let $x,y\in U$ be given with $x\sim_{G'}y.$ We have elements
$g_1',g_2'$ in $G'$ with $g_1'x=g_2'y.$ We furthermore have
elements $h_1,h_2$ in $G$ with $h_1g_1'=g_1\in G$ and
$h_2g_2'=g_2\in G.$ Now
$$g_1h_2x=h_1h_2g_1'x=h_1h_2g_2'y=h_1g_2y.$$
Thus $x\sim_G y.$ This proves $E(G')\subset E(G);$ hence
$E(G)=E(G').$
\end{proof}

\begin{defn}\label{defn8.6} We call $G'$ the \bfem{saturation} of
the monoid $G$ (in $U),$ and we say that $G$ is saturated if
$G=G'.$\end{defn}

It is immediate that $(G')'=G'.$ Thus $G'$ is always saturated.

\begin{examp}\label{examp8.7} If $S(U)$ happens to be a group,
then the saturation of a submonoid $G$ of $S(U)$ is just the
subgroup of $S(U)$ generated by $G.$ Indeed, the elements of $G'$
are the $x\in S(U)$ with $g_1x=g_2$ for some $g_1,g_2\in G,$ i.e.,
the elements $g_1^{-1}g_2$ with $g_1,g_2\in G.$\end{examp}

\begin{prop}\label{prop8.8}
Let $E$ be a multiplicative equivalence relation on $U.$
\begin{enumerate}\item[a)] The set
$$G_E:=\{x\in S(U)\bigm|x\sim_E1\}$$ is a saturated submonoid of
$S(U).$

\item[b)] If $E=E(H)$ for some submonoid $H$ of $S(U)$,
then $G_E$ is the saturation $H'$ of $H.$

\item[c)] In general,
$E(G_E)$ is the coarsest orbital equivalence relation on $U$ which
is finer than~$E.$

\item[d)] If $E$ is MFCE then $G_E\subset
S_e(U),$ and $E(G_E)$ is the coarsest OFCE-relation on $U$ which
is finer than $E.$
\end{enumerate}
\end{prop}

\begin{proof} a): If $x,y\in G_E$ then $x\sim_E1,$ $y\sim_E1;$
hence $xy\sim_Ey\sim_E1,$ thus $xy\in G_E.$ This proves that $G_E$
is a submonoid of $S(U).$ Let $x\in G_E'$ be given. We have
elements $g,h\in G_E$ with $hx=g.$ It follows from $g\sim_E1,$
$h\sim_E1$ that
$$x\sim_Ehx=g\sim_E1.$$
Thus $x\in G_E.$ This proves that $G_E'=G_E.$

b): Assume that $E=E(H)$ with $H$ a submonoid of $S(U).$ For $x\in
S(U)$ we have
$$x\sim_E1\quad\Leftrightarrow\quad \exists h_1,h_2\in H:
h_1x=h_2\quad\Leftrightarrow\quad x\in H'.$$ Thus $G_E=H'.$

c): Let $G:=G_E.$ If $x\sim_Gy$ then $g_1x=g_2y$ with some
$g_1,g_2\in G.$ From $g_1\sim_E1,$ $g_2\sim_E1,$ we conclude that
$$x\sim_E g_1x=g_2y\sim_Ey.$$
Thus $E(G)\subset E.$ If $H$ is any submonoid of $S(U)$ with
$E(H)\subset E,$ then
$$H\subset G_{E(H)}\subset G_E=G.$$
Thus $E(H)\subset E(G).$

d): Assume that $E$ is MFCE. If $x\in G_E$ then we conclude from
$x\sim_E1$ that $ex=e.$ Thus $G_E\subset S_e(U).$ Every
multiplicative equivalence relation on $U$ which is finer than $E$
is MFCE. In particular, this holds for orbital relations. We learn
from c) that $E(G_E)$ is the coarsest OFCE-relation on $U$ finer
than $E.$
\end{proof}

We denote the set of saturated submonoids of $S(U)$ by Sat$(S(U))$
and the set of saturated submonoids of $S_e(U)$ by Sat$(S_e(U)).$

\begin{schol}\label{schol8.9} Propositions \ref{prop8.5} and
\ref{prop8.8}
 imply that we have an isomorphism of posets $H~\mapsto~E(H)$ from  ${\rm Sat}(S(U))$ to
 $\Orb(U),$
 mapping ${\rm Sat}(S_e(U))$ onto $\OFC(U),$ with inverse map $E\mapsto G_E.$ $\{$Here, of course, both sets
${\rm Sat}(S(U))$ and $\Orb(U)$ are ordered by
inclusion.$\}$\end{schol}

It is fairly obvious that Sat$(S(U))$ is a complete lattice.
Indeed, the supremum of a family $(H_i\bigm|i\in I)$ of saturated
submonoids of $S(U)$ is the saturation $H'$ of the submonoid of
$S(U)$ generated by the $H_i,$ while the infimum of this family is
the saturation $(\bigcap_i H_i)'$ of the intersection of the
family. Thus also $\Orb(U)$ is a complete lattice. It follows that
Sat$(S_e(U))$ and $\OFC(U)$ are complete sublattices of
Sat$(S(U))$ and $\Orb(U)$, respectively.

Let Mult$(U)$ denote the set of all multiplicative equivalence
relations on $U,$ partially ordered by inclusion.  In \S7 we have
seen that the subposet $\MFC(U)$ of Mult$(U),$ consisting of the
MFCE-relations on $U,$ is a complete lattice. In the same way one
proves that Mult$(U)$ itself is a complete lattice, the supremum
and infimum of a family in Mult$(U)$ being given in exactly the
same way as in \S7 for MFCE-relations. This makes it also evident
that $\MFC(U)$ is a complete sublattice of Mult$(U).$

We doubt whether $\Orb(U)$ and $\OFC(U)$ are always sublattices of
Mult$(U)$ and  $\MFC(U)$, respectively. But we have the following
partial result.

\begin{prop}\label{prop8.10} Let $(G_i\bigm|i\in I)$ be a family
of submonoids of $S(U)$, and let $G$ denote the monoid generated
by this family in $S(U).$ Then, in the lattice ${\rm Mult}(U),$
$$E(G)=\bigvee_{i\in I} E(G_i).$$
$\{$N.B. Thus the same holds in $\MFC(U)$, if every $G_i\subset
S_e(U).\}$
\end{prop}

\begin{proof} Let $F:=\bigvee_i E(G_i)$ in Mult$(U).$ Of course,
$F\subset E(G)$ since each $E(G_i)\subset E(G).$ Let $x,y\in U$ be
given with $x\sim_Gy.$ We want to conclude that $x\sim_Fy,$ and
then will be done.

We have $gx=hy$ with elements $g,h$ of $G.$ Now $g$ and $h$ are
products of elements in $\bigcup_i G_i,$ and for any $g'\in
\bigcup_i G_i$ and $z\in U$, we have $z\sim_Fg'z.$ It follows that
$x\sim_Fgx$ and $y\sim_Fhy;$ hence $x\sim_Fy.$
\end{proof}

We present an important case where $\OFC(U)$ and $\MFC(U)$ nearly
coincide.

\begin{thm}\label{thm8.11} Assume that every $x\in \mathcal T(U)$ is
invertible; hence $\mathcal T(U)$ is a group under multiplication.
$\{$The main case is that $U$ is a supertropical semifield.\} Let
$E$ be an MFCE-relation on $U.$ Then either $E=E(\nu),$ i.e., $E$
is the top element of $\MFC(U)$ (cf.~Example~\ref{examps6.4}.ii),
or $E$ is orbital.
\end{thm}

\begin{proof} a) Assume that there exists some $x_0\in \mathcal T(U)$ with
$x_0\sim_Eex_0.$ Multiplying by $x_0^{-1}$ we obtain $1\sim_Ee,$
and then obtain $x\sim_Eex$ for every $x\in U.$ Thus $E=E(\nu).$

b) Assume now that $x\not\sim_E ex$ for every $x\in \mathcal T(U)$
(i.e., $E\subset \tE). $ Clearly $S_e(U)=\mathcal T_e(U)$. Let
$$H:=G(E)=\{x\in \mathcal T(U)\bigm| x\sim_E1\}.$$
Then $E(H)\subset E.$ Given $x,y\in U$ with $x\sim_Ey,$ we want to
prove that $x\sim_Hy.$ We have $ex=ey.$ If $x\in eU$ or $y\in eU$,
we conclude that $x=y,$ due to our assumption on $E.$ There
remains the case that both $x$ and $y$ are tangible. Then we infer
from $x\sim_Ey$ that
$$1=x^{-1}x\sim_Ex^{-1}y.$$
Thus $x^{-1}y\in H,$ which implies $x\sim_Hy.$ This completes the
proof that $E=E(H).$
\end{proof}

\begin{cor}\label{cor8.12}
If every element of $\mathcal T(U)$ is invertible, then the poset
$\MFC(U)\setminus\{E(\nu)\}$ is isomorphic to the lattice of
subgroups of $\mathcal T_e(U).$\end{cor}

\begin{prop}\label{prop8.13} If $R$ is a semifield, then every supervaluation $\vrp: R \to
U$ with $U \neq e U$ is tangible.

\end{prop}
This follows from Theorem \ref{thm8.11} applied to the target
$U(v)$ of the initial supervaluation  $\vrp_v$ of $v := e \vrp $,
since for any orbital equivalence relation $E$ on $U(v)$ the
transmission $\pi_E$ sends tangibles to tangibles. A more direct
proof runs as follows.

\begin{proof} Let $a \in R $, $a \neq 0$. Then
$$ \vrp(a) \vrp(a^{-1}) = \vrp (1) = 1.$$
Since $1_U \neq e_U$ this forces $\vrp(a)$ to be tangible.
\end{proof}
N.B. The argument shows more generally that any supervaluation on
a semiring sends units to tangible elements,provided not the whole
target is ghost. \vskip 3mm

In the case that $R$ is a field the following in now amply clear.
\begin{schol}\label{schol8.13}
If $v$ is a Krull valuation on a field $R$ with value group
$\Gamma$, then the lattice $\Cov(v)$ of equivalence classes of
supervaluations covering $v$ is anti-isomorphic to the lattice of
subgroups of the unit group $\mathfrak o_v^*$ of the valuation
domain $\mathfrak o_v:=\{x\in R\bigm| v(x)\le 1\},$ augmented by
one element at the top.
\end{schol}

\section{The \UTGR; strong supervaluations}\label{sec:10}

Let $U$ be any supertropical semiring. If $x,y \in U$, it has
become customary to write
$$ x = y + \text{ghost}$$
if $x$ equals $y$ plus an unspecified ghost element (including
zero). In more formal terms we have a binary relation $\lmodg$ on
$U$ defined as follows:

\begin{defn}\label{defn10.1}
$$x \lmodg y \ \Leftrightarrow \ \exists z \in e U \ \text{with} \ x = y
+z.$$ We call $\lmodg$ the \textbf{ghost surpassing relation} on
$U$ or \textbf{\UTG-relation}, for short.

\end{defn}

The \UTG-relation seems to be at the heart of many supertropical
arguments. Intuitively $x \lmodg y$  means that $x$ coincides with
$y$ up to some ``negligible'' or ``near-zero'' element,  namely a
ghost element. But we have to handle the \UTG-relation with care,
since it is not symmetric. In fact it is antisymmetric, see below.

The \UTG-relation is clearly transitive:
$$ x \lmodg y, \ y \lmodg z \ \ \Rightarrow  \  x \lmodg z.$$
It is also compatible with addition and multiplication: For any $z
\in U$, $x \lmodg y$ implies $x + z \lmodg y + z$, and $x  z
\lmodg y z$.

We observe the following further properties of this subtle binary
relation.

\begin{remark}\label{rem10.2} Let $x,y \in U$.
\begin{enumerate} \eroman
    \item  $x = y  \ \Rightarrow  \ x \lmodg y  \ \Rightarrow \ \nu(x ) \geq
    \nu(y)$.
    \item  If $x \in \tT(U) \cup \{ 0 \}$, then
$x \lmodg y  \ \Leftrightarrow \ x = y.
             $

\item If $x \in \tG(U) \cup \{ 0 \}$, then
$ x \lmodg y  \ \Leftrightarrow \ \nu(x) \geq \nu(y).
$

\item $x \lmodg 0  \ \text{iff } \ x = eU$.

\end{enumerate}

\end{remark}

\begin{lemma}\label{lem10.3}
The \UTG-relation is antisymmetric, i.e.;
$$ x \lmodg y,  \ y \lmodg x \ \Rightarrow  \ x  =  y.
$$
\end{lemma}

\begin{proof}

If $x \in \tT(U)$ or  $y \in \tT(U)$ this is clear by Remark
\ref{rem10.2}.ii. Assume now that both $x,y \in eU$. Then $\nu(x)
\geq \nu(y)$ and $\nu(y) \geq \nu(x)$ by Remark \ref{rem10.2}.iii;
hence $\nu(x) =  \nu (y)$, i.e., $x = y$.
\end{proof}

\begin{prop}\label{prop10.4} $ $

\begin{enumerate} \eroman
    \item  Assume that $\al:U \to V$ is a transmission. Then, for any $x,y \in U$,
    $$ x \lmodg y \ \Rightarrow \ \al(x) \lmodg \al(y).$$

    \item Assume that $\vrp :R  \to U$  and $\psi: R \to V$  are
    supervaluations with $\vrp \geq \psi$. Then for any $a,b \in R$
    $$ \vrp(a) \lmodg \vrp(b) \ \Rightarrow \ \psi(a) \lmodg \psi(b).$$
\end{enumerate}

\end{prop}

\begin{proof} i): Let $x \lmodg y $. If $x$ is tangible or zero,
then $x= y$; hence $\al(x) = \al(y)$. If $x$   is ghost, then
$\nu(x) \geq \nu(y)$; hence  $$\nu(\al(x)) = \al(\nu(x)) \geq
\al(\nu(y)) = \nu(\al(y))$$ by rule TM5 in \S5. Since $\al(x)$  is
ghost, this means $\al(x) \lmodg \al(y)$, cf.~Remark
\ref{rem10.2}.iii above.

ii): We may assume that the supervaluation $\vrp$  is surjective.
By \S5 we have a (unique) transmission $\al: U \to V$ with $\al
\circ \vrp = \psi$. Thus the claim follows  from part i).
\end{proof}

We cannot resist giving  a second proof of part ii) of the
proposition relying only on Definition \ref{defn5.1} of dominance
(conditions D1-D3).

\begin{proof}[Second proof of Proposition \ref{prop10.4}.ii] Assume
that $\vrp(a) \lmodg \vrp(b)$. If $\vrp(a)$   is tangible or zero,
then $\vrp(a) = \vrp(b)$;  hence $\psi(a) = \psi(b)$ by D1; hence
$\psi(a) \lmodg \psi(b)$.
 If
$\vrp(a)$ is ghost then $e\vrp(a) \geq e\vrp(b)$; hence $e\psi(a)
\geq e\psi(b)$ by D2. By D3 the element $\psi(a)$  is ghost. Thus
$\psi(a) \lmodg \psi(b)$ again,
\end{proof}


The \UTG-relation seems to be helpful for analyzing additivity
properties of supervaluations.

\begin{lemma} \label{lem10.5} If $\vrp:R \to U$ is a
supervaluation on a semiring $R$ with $\vrp(a) + \vrp(b) \in eU$,
then
  \begin{equation}\renewcommand{\theequation}{$*$}\addtocounter{equation}{-1}\label{eq:str.1}
\vrp(a) + \vrp(b) \  \lmodg  \ \vrp(a+b).
\end{equation}
\end{lemma}

\begin{proof}
Let $v: R \to eU$ denote the \m-valuation covered by $\vrp$, $v =
e\vrp$. We have $v(a+b) \leq v(a) + v(b)$; hence   $e\vrp(a+b)
\leq e(\vrp(a) + \vrp(b))$. If $\vrp(a) + \vrp(b) \in eU$, this
shows that $\vrp(a) + \vrp(b) \lmodg \vrp(a+b)$.
\end{proof}

It will turn out to be desirable to have supervaluations on $R$ at
hand, where the property $(*)$ holds for \textbf{all} elements
$a,b$ of $R$.

\begin{defn} \label{defn10.6}
We call a supervaluation $\vrp:R \to U$  \textbf{tangibly
additive}, if in addition to the rules SV1-SV4 from \S4 the
following axiom holds:

 \begin{alignat*}{2}
& SV5: \ && \text{If } a,b \in R \ \text{and } \vrp(a) + \vrp(b)
\in \tT(U), \ \text{then } \vrp(a) + \vrp(b) = \vrp(a+b).
\end{alignat*}
\end{defn}

\begin{prop}\label{prop10.7} A supervaluation $\vrp: R \to U$ is
tangibly additive iff for any $a,b \in R$
$$ \vrp(a) + \vrp(b)  \ \lmodg \ \vrp(a+b).$$
\end{prop}
\begin{proof}
This is clear by Lemma \ref{lem10.5} and Remark \ref{rem10.2}.ii
above.
\end{proof}

\begin{cor}\label{cor10.8}
If $\vrp:R \to U$ is tangibly additive, then for every finite
sequence $a_1, \dots, a_m$ of elements of $R$
$$ \sum_{i=1}^m \vrp(a_i) \lmodg \vrp \bigg(\sum_{i=1}^m a_i\bigg).$$
\end{cor}

\begin{proof}
This holds for $m=2$  by Proposition \ref{prop10.7}. The general
case follows by an easy induction using the transitivity  of the
\UTG-relation.
\end{proof}

\noindent \emph{Comment:}  We elaborate  what it means that a
given supervaluation $\vrp: R \to U$  is tangibly additive in the
case that the underlying \m-valuation $v = e\vrp: R \to eU$ is
{strong}.

Let $a,b \in R $ be given with $\vrp(a)+ \vrp(b) \in \tT(U)$,
i.e., $v(a) \neq v(b)$, and assume without loss of generality that
$v(a) < v(b)$. Then $v(a+b) = v(b)$. Hence, $\vrp(a+b)$   is some
element of the fiber $\nu^{-1}_U(v(b))$; but the axioms SV1-SV4
say little about the position of $\vrp(a+b)$ in this fiber. SV5
demands that $\vrp(a+b)$ has the ``correct'' value $\vrp(a) +
\vrp(b) = \vrp(b)$. \hfill\quad\qed

Concerning   applications the strong  \m-valuations seem to be
more important than the others. (Recall that any \m-valuation on a
ring is strong.) Thus the tangibly additive supervaluations
covering strong \m-valuations  deserve a name on their own.

\begin{defn}\label{defn10.9} We call a \textbf{supervaluation $\vrp: R \to
U$ strong} if $\vrp$ is tangibly  additive and the covered
\m-valuation $e \vrp: R \to eU$ is strong.
\end{defn}

We exhibit an important case where a tangibly additive
supervaluation is automatically strong.

\begin{prop}\label{prop10.10} Assume that $\vrp:R \to U$ is a
 tangible (cf.~Definition \ref{defn4.1}) and tangibly additive
supervaluation. Then $\vrp$ is strong.
\end{prop}

\begin{proof}
We have to verify that $v:= e \vrp$  is strong. Let $a,b \in R$ be
given with $v(a) \neq v(b)$. Suppose without loss of generality
that $v(a) < v(b)$. Then $\vrp(a),\vrp(b) \in U$ and $\vrp(b) \neq
0$. Since $\vrp$ is tangible, $\vrp(b) \in \tT(U)$. It follows
that $\vrp(a) + \vrp(b) \in \tT(U)$; hence
$$ \vrp(a) + \vrp(b) = \vrp (a + b),$$
because  $\vrp$ is tangibly additive. Multiplying  by $e$ we
obtain
$$ v(a) + v(b) = v (a + b).$$
\end{proof}

We  now are  ready to aim at an application of the supervaluation
theory developed so far. We start with the polynomial semiring
$R[\lm] = R [\lm_1, \dots, \lm_n]$ in  a sequence $\lm = (\lm_1,
\dots, \lm_n)$ of $n$ variables over a semiring $R$. Let $\vrp: R
\to U$  be a tangibly additive valuation with underlying
\m-valuation $v: R \to M$, $M:= eU$.

Given a polynomial \begin{equation}\label{poly1} f = \sum_i c_i
\lm^i \in R[\lm]\end{equation} in the usual multimonomial notation
($i$ runs though the multi-indices $i = (i_1, \dots, i_n) \in
\N_0^{n}$, $\lm^i = \lm_1^{i_1} \cdots \lm_n^{i_n}$, only finitely
many $c_i \neq 0$), we obtain from $f$ polynomials
$$ \tvrp(f) \ : = \ \sum _i \vrp(c_i) \lm ^i \in U[\lm],$$
$$ \tv(f) \ : = \ \sum _i v(c_i) \lm ^i \in M[\lm],$$
by applying $\vrp$ and $v$ to  the coefficients of  $f$. This
gives us maps
$$ \tvrp :  \ R[\lm] \ \to \ U[\lm], \quad \tv :  \ R[\lm] \ \to \ M[\lm].$$

Let $a = (a_1,\dots,  a_n \in R^n)$ be an $n$-tuple of elements of
$R$. It gives us $n$-tuples
$$ \vrp(a) = (\vrp(a_1), \dots, \vrp(a_n) ), \quad v(a) = (v(a_1), \dots, v(a_n) )$$
in $U^n$ and $M^n$, respectively. We have an evaluation map
$\ep_a: R[\lm] \to R,$ which sends the polynomial $f$ (notation as
in \eqref{poly1}) to
\begin{equation}\label{poly2} \ep_a(f) \  =  f(a) \ = \  \sum _i c_i a ^i
\end{equation}
and analogous evaluation maps
$$ \ep_{\vrp(a)}(f)  : \  U[\lm]  \to U,   \quad    \ep_{v(a)}(f)  : \  M[\lm] \to M. $$
These evaluation maps are semiring homomorphisms. We have a
diagram
$$\xymatrix{
           R[\lm] %
            \ar[d]_{\tvrp}     \ar[rr]^{\ep_a}    &&      R
       \ar[d]^{\vrp}
       \\
      U[\lm]    \ar[rr]^{\ep_{\vrp(a)}}   &&        U
 }$$
(and an analogous diagram with $v$ instead of $\vrp$) which
usually does not commute. But it commutes ``nearly''.

\begin{thm}\label{thm10.11} For $f \in R[\lm]$
$$ \ep_{\vrp(a)}(\tvrp(f)) \ \lmodg \ \vrp(\ep_a(f)).$$
\end{thm}
\begin{proof}
Let again $f = \sum_i c_i \lm^i$. Now $\vrp(\ep_a(f)) =
\vrp(\sum_i c_i a^i)$, while $$\ep_{\vrp(a)(\tvrp(f))} = \sum_i
\vrp(c_i) \vrp(a)^i = \sum_i \vrp(c_ia^i).$$ Thus the claim is
that
  \begin{equation}\renewcommand{\theequation}{$*$}\addtocounter{equation}{-1}\label{eq:str.1}
\sum_i \vrp(c_i a^i) \  \lmodg  \ \vrp(\sum_i c_i a^i).
\end{equation}
This follows  from Corollary \ref{cor10.8} above. \end{proof}

We draw a consequence of this theorem. Let $$Z(f) \ := \ \{ a \in
R^n \ | \ f(a) = 0\}, $$ the zero set of $f$. Let further
$$ Z_0(\tvrp(f)) \ :=  \ \{ b \in U^n \ | \ \tvrp(f)(b) \in eU \}, $$
which we call the \textbf{root set} of $\tvrp(f)$. For $a \in
Z(f)$ we have $\vrp\left( \sum_i c_i a^i \right) = 0$. It follows
by Theorem \ref{thm10.11} that $\tvrp(f)(\vrp(a)) \lmodg 0$, i.e.,
$\tvrp(f)(\vrp(a))$ is ghost.

We have proved
\begin{cor}\label{cor10.12} If $\vrp:R \to U$  is tangibly
additive, then, for any $f \in R[\lm]$,
$$ \vrp(Z(f)) \subset Z_0(\tvrp(f)).$$
 \hfill\quad\qed

\end{cor}

Assume now that $\vrp$ is \emph{tangible} and tangibly additive;
hence strong (cf.~Proposition \ref{prop10.10}). Then,  of course,
$\vrp(Z(f)) \subset \tT(U)_0^n$ with   $\tT(U)_0 := \tT(U) \cup \{
0 \}$. Thus we have
  \begin{equation}\renewcommand{\theequation}{$**$}\addtocounter{equation}{-1}\label{eq:str.1}
\vrp(Z(f)) \ \subset \ Z_0(\tvrp(f))_{\tan} \
\end{equation}
with
$$Z_0(\tvrp(f))_{\tan} \ : = \ Z_0(\tvrp(f)) \cap \tT(U)^n_0, $$
which we call \textbf{tangible root set} of $\tvrp(f)$. We want to
translate $(**)$ into a statement about the relation between
$Z(f)$ and the so called ``corner locus'', of the polynomial
$\tv(f) \in M[\lm]$, to be defined.

We call a polynomial $g = \sum_i d_i \lm^i \in M[\lm]$  a
\textbf{tropical polynomial}, and define the \textbf{corner-locus}
$\corn(g)$ of $g$ as the set of all $b \in M^n$ such that there
exists two different multi-indices  $j,k \in \N_0^n$ with
$$ d_j b^j = d_k b^k \geq d_i b^i$$
for all $i \neq j,k$. We also say that $\corn(g)$ is the
\textbf{tropical hypersurface} defined by the tropical polynomial
$g$.

This is well established terminology at least in the ``classical
case'' that $M$ is the bipotent semiring $T(\R)$ given by the
order monoid $(\R,+ \ )$, the so called max-plus algebra of $\R$
(cf.~\S1, \cite[\S1.5]{IMS}). A small point here is, that we admit
coordinates  with value $0_M := -\infty$, which usually is not
done in tropical geometry.  On the other hand we could work as
well with Laurent polynomials. Then of course we would have to
discard the zero element.

Returning to our tangible strong supervaluation $\vrp: R \to  U$
and the \m-valuation  $$v = e \vrp : R \to M,$$  we look at the
tropical polynomial
$$ \tv(f) = \sum_i v(c_i)\lm^i$$
from above. Let $a \in R^n$. Then
$$ \tvrp(f)(\vrp(a)) \ =  \ \sum \vrp(c_i a^i),$$
and all summands are the right side are in $\tT(U)_0$. Thus the
sum is ghost iff the maximum of the $\nu$-values
$$ \nu(\vrp(c_i a^i)) = v(c_i) v(a^i) \qquad (i \in \N^n_0)$$
is attained for at least two multi-indices. This means that $v(a)
\in \corn(\tv(f))$.

Thus $(**)$ has the following consequence

\begin{cor} \label{cor10.13}
Let $v:R \to M$ be  a strong \m-valuation   on a semiring $R$.
Assume that there exists a tangible supervaluation $\vrp:R \to U$
covering $v$. Then for any polynomial $f \in R[\lm]$,
$$ v(Z(f)) \ \subset \ \corn( \tv(f)).$$
\end{cor}

We have arrived at a very general version of the Lemma of Kapranov
(\cite[Lemma 2.1.4]{EKL}), as soon as we find a tangible cover
$\vrp: R \to U$ of the given \m-valuation $v: R \to M$.  This
turns out to be easy in the case that $M$ is cancellative (i.e.,
$v$ is a strong \emph{valuation}).

\begin{lemma}\label{lem10.14} Suppose there is given a \textbf{tangible
multiplicative section} of the ghost map $\nu : U \to M$, i.e., a
map $s: M \to \tT(U)_0$ with $s(0) = 0$, $s(1) = 1$, $s(xy) =
s(x)s(y)$, and $\nu(s(x))=x$ for any $x,y \in M$. Let $v:R \to M$
be a strong \m-valuation. Then $s \circ v : R \to U$ is a tangible
strong supervaluation covering $v$.
\end{lemma}

\begin{proof}
Clearly  $\vrp = sv$ obeys SV1-SV4. Let  $a,b \in R$ be given with
$v(a) <
 v(b)$.  Then $v(a+b) = v(b)$; hence $s v (a+b) = sv(b) $. Thus
 SV5 holds true. We have $e \vrp = \nu \circ \vrp = v$.
\end{proof}

\begin{example}
 If $U$ is a supertropical semifield, it is known that such a
section $s$ always exists (\cite[Proposition
1.6]{IzhakianRowen2009Equations}).
\end{example}

\begin{example}\label{examp10.5}
Assume that $M$ is a cancellative bipotent semiring, and $v: R \to
M$ is a strong valuation. We take $U:= D(M \setminus \{ 0\})$
(Example \ref{examps3.16}), for which we write more briefly
$D(M)$. For every $z\in M$  there exists a unique $x \in \tT(U)_0$
with $\nu(x) = z$. We write $x = \hat z$. Clearly $z \mapsto \hat
z $ is a tangible multiplicative section of the ghost map, in fact
the only one.  By the lemma we obtain a tangible supervaluation
$$ \hat v : R \to U, \quad \hat v (z) := \widehat{v(z)},$$
which covers $v$, in fact the only such supervaluation.
\end{example}

Looking again at Corollary \ref{cor10.13} we now know that
$$ v(Z(f)) \subset \corn(\tv(f)),$$
whenever $v: R \to M$ is a strong valuation and $f\in R[\lm]$.

\section{The tangible strong supervaluations in $\Cov(v)$} \label{sec:11}

Given an \m-valuation $v: R \to M$, recall from \S7 that the
equivalence classes  $[\vrp]$ of supervaluations $\vrp$  covering
$v$ form a complete lattice $\Cov(v)$. Abusing notation, we
usually will not distinguish  between a supervaluation $\vrp$  and
its class $[\vrp]$, thus writing $\vrp \in \Cov(v)$ if $\vrp$
covers $v$. This will cause no harm in the present section. \{N.B
If you are sceptical   about this,  you may always assume that
$\vrp$  is surjective, more specially, that $\vrp = \vrp_v / E$
with $\vrp_v$ the initial covering of $v$ and $E$ an
\MFCE-relation on $U(v)$ (cf.~Notation \ref{not6.14}). These
supervaluations $\vrp$ are canonical representatives  of their
classes $[\vrp]$.\}

\begin{lemma}\label{lem11.1}
Assume that $\vrp:R \to U$ and $\psi:R \to V $ are supervaluation
with $\vrp \geq \psi.$
\begin{enumerate}\eroman
    \item  If $\psi$ is tangible, then $\vrp$  is tangible.
    \item  It $\vrp$ is tangibly additive, then $\psi$   is
    tangibly additive.
\end{enumerate}
\end{lemma}

\begin{proof}
i): is clear from the axiom D3 in the definition of dominance
(cf.~Definition \ref{defn5.1}).

ii): follows from Propositions \ref{prop10.7} and
\ref{prop10.4}.ii.
\end{proof}

\emph{Starting from now we assume that $v$ is a strong valuation}
(which means in particular that $M$ is cancellative). Let $\gq$
denotes the support of $v$, i.e., $\gq = v^{-1}(0)$.

\begin{notation}\label{not11.2} $\Cov_{\tng}(v)$ denotes the set of
tangible supervaluations in $\Cov(v)$, and $\Cov_{\stg}(v)$
denotes the set of strong ($=$ tangibly additive) supervaluations
in $\Cov(v)$. Finally, let $$ \stCov(v) \ := \ \tCov(v ) \cap
\sCov(v),
$$
be the set of tangible strong supervaluations covering $v$.
\end{notation}

We already know by Example \ref{examp10.5} that the set
$\stCov(v)$ is not empty. Lemma \ref{lem11.1} tells us in
particular that $\tCov(v)$ is an upper set and $\sCov(v)$ is a
lower set in the poset $\Cov(v)$.

Let us study these sets more closely. We start with $\tCov(v)$.
The initial supervaluation $\vrp_v: R \to U(v)$ (cf.~Definition
\ref{defn5.15}) is the top ($=$ biggest) element of $\Cov(v)$, and
thus is also the top element of $\tCov(v)$.  This can also be read
off from the explicit description of $\vrp_v$ in Example
\ref{examp4.5}. The other elements of $\Cov(v)$  are the
supervaluations $\vrp_v / E : R \to U(v) / E$, with $E$ running
through the \MFCE-relations on $U(v)$. We have to find out which
\MFCE-relations $E$ on $U(v)$ give tangible supervaluations
$\vrp_v / E$.

Here is a definition which - for later use - is slightly more
general than what we need now:

\begin{defn}\label{defn11.3} We call an equivalence relation $E$
on a supertropical semiring $U$ \textbf{ghost separating} if for
all $x \in \tT (U)$, $y \in U$,
$$ x \sim_E y \quad \Rightarrow \quad y \in \tT(U) \ \text{or} \ x \sim_E 0.$$
\end{defn}

If $E$ is an \MFCE-relation on $U$, then $x \sim_E 0$ only if
$x=0$. Thus, $E$ is ghost separating iff $\tT(U)$  is a union of
$E$-equivalence classes. This means that $E$ is finer than the
\MFCE-relation $\tE$ introduced in Examples \ref{examps6.4}.v,
whose equivalence classes are the tangible fibers of $\nu_U$ and
the one-point sets in $eU$.

If $\vrp:R \to U$ is a surjective tangible supervaluation and $E$
is an \MFCE-relation on $U$, then it is obvious that $\vrp / E : R
\to U / E$ is again tangible iff $E$ is ghost separating. Thus we
see that $\vrp_v / \tE$ is the bottom ($=$ smallest) element of
$\tCov(v)$.

Now recall from Example \ref{examp6.12} that, in the notation at
the end of \S9 (Example \ref{examp10.5}),
$$ U(v) / \tE \ = \ D(M) ;$$
hence $\vrp_v / \tE$ coincides with the only tangible cover $\hat
v$ of $v$ with values in $D(M)$, cf.~Example~\ref{examp10.5}. We
conclude  that
$$ \tCov (\vrp) = \{ \psi \in \Cov(v) \ | \ \psi \geq \hat v\}.$$
Again by Example~\ref{examp10.5} we know that $\hat v$ is strong.
This $\hat v$ is also the bottom of the poset~$\stCov(\vrp)$.

We turn to $\sCov(v)$. We will construct a new  element of this
poset in a direct way. For that reason we introduce an equivalence
relation on $R$.

\begin{defn}\label{defn11.4} Let $S(v)$ denote the  equivalence
relation on the set $R$ defined as follows. \{We write $\sim_v$
for $\sim_{S(v)}$.\}

If $a_1,a_2 \in R$ then
$$\begin{array}{lll}
   a_1 \sim_v a_2  & \Longleftrightarrow & \text{either } v(a_1) = v(a_2) = 0 \\
     &  &  \text{or } \exists c_1, c_2 \in R, \ \text{with } v(c_1) < v(a_1), \  v(c_2) < v(a_2),\\
     & & a_1 + c_1 = a_2 + c_2.
  \end{array}$$
\end{defn}
It is easily checked that $S(v)$ is indeed an equivalence relation
on the set $R$, by making strong use if the assumption that the
valuation $v$ is strong. This is the finest equivalence relation
$E$ on $U$ such that $a \sim_E a +c $  if $v(c) < v(a)$. Observe
also that
$$ a_1 \vsim  a_2 \ \Longrightarrow \ v(a_1) = v(a_2).$$

We claim that  $S(v)$ is compatible with  multiplication, i.e.,
$$ a_1 \vsim  a_2 \ \Longrightarrow \ a_1b \vsim  a_2b$$
for every $b \in R$. This is obvious if $a_1 \in \gq$  or $a_2 \in
\gq$, or $b \in \gq$. Otherwise $v(b)> 0$, and we have elements
$c_1,c_2 \in R$ with $ v(c_1) < v(a_1)$, $v(c_2) < v(a_2)$, $a_1 +
c_1 = a_2 + c_2$. Then $a_1b + c_1b = a_2b + c_2b$ and
$$v(c_i b ) \ = \  v(c_i)v(b) \ < \ v(a_i)v(b) \ = \ v(a_ib)$$
for $i = 1,2$, since by assumption $M$ is cancellative. Thus
indeed $a_1 b \vsim a_2b$.

We denote the $S(v)$-equivalence class of an element $a$ of $R$ by
$[a]_v$. The set  $\bR := R / S(v)$ is a monoid under the well
defined multiplication $$ [a]_v \cdot [b]_v \ = \ [ab]_v$$ for
$a,b \in R$. The subset $R \setminus \gq$ of $R$ is a union of
$S(v)$-equivalence classes and the subset $\overline{R \setminus
\gq} : = (\Rmq) / S(v)$ of $\bR$ is a submonoid of $R$. We have
$$\bR = \overline{R \setminus \gq} \  \cup \  \{ \bar 0 \}$$ with
$\bar 0 = [0]_v = \gq.$

Since $a_1 \vsim a_2$ implies $v(a_1) = v(a_2)$, we have a well
defined monoid homomorphism $\bR \to M$, $[a]_v \mapsto v(a)$,
which restricts to a  monoid homomorphism $$ \bv: \overline{R
\setminus \gq} \ \to \ M \setminus \{ 0\}.$$ This map $\bv$  gives
us a supertropical semiring
$$ U := \STR(\overline{R
\setminus \gq}, M \setminus \{ 0\}, \bv),$$ cf. Construction
\ref{constr3.14}. Notice that $\tT(U) = \overline{R \setminus
\gq}$ and $eU = M$. We identify $\tT(U)_0 = \bR$.

\begin{prop}\label{prop11.5} The map $\chi: R \to U$ given by $$\chi(a) :=
0 \quad \text{if}\quad a \in \gq, \qquad \chi(a) := [a]_v \in
\tT(U) = \bRmq \quad\text{if} \quad a \notin \gq,$$ is a tangible
strong supervaluation covering $v$.

\end{prop}
\begin{proof} It is obvious that $\chi$ obeys the rules SV1-SV3 in
the definition of supervaluations (Definition \ref{defn4.1}). Due
to our  construction of $U$ we have $\nu_U \circ  \chi = v$. Thus
$\chi$ also obeys SV4, and hence is a supervaluation covering the
strong valuation $v$. It is clearly tangible.

It remains to verify that  $\chi$ is tangibly additive. Let $a,b
\in R$ be given with $\chi(a) + \chi(b) \in \tT(U)$, i.e., $v(a)
\neq v(b)$. Assume without loss of generality that $v(a) < v(b)$.
Then $a + b \vsim b$. This means that $\chi(a+b) = \chi(b)$, as
desired.
\end{proof}

We strive for an understanding of the set of all $\psi \in \Cov
(v)$ which are dominated by this supervaluation $\chi$. We need a
new definition.

\begin{defn}\label{defn11.6}
We call a supervaluation $\vrp: R \to V$ \textbf{very strong}, if
 \begin{alignat*}{2}
&SV5^*:\ && \forall a,b \in R : e\vrp(a)< e \vrp(b) \
\Longrightarrow \ \vrp(a+b) = \vrp(b).
\end{alignat*}
\end{defn}
 Clearly SV$5^*$ implies that the \m-valuation $v$ is strong. If we
 require this property only for $a,b \in R$ with $e\vrp(a) < e \vrp(b)$
and  $\vrp(b)$ tangible, we are back to condition SV5  given above
(Definition \ref{defn10.6}). Thus, a very strong supervaluation is
certainly strong.  On  the other hand, every \textbf{tangible}
strong supervaluation is very strong.

\begin{lem}\label{lem11.7}
If $\vrp: R \to V$ is very strong, then any supervaluation $\psi:
R \to W$ dominated by $\vrp$ is again very strong.
\end{lem}
\begin{proof} Let $a,b \in R $ be given with $ e\psi(a) < e \psi
(b)$. It follows from axiom D2 that $ e\vrp(a) < e \vrp (b)$,
since $ e\vrp(a) \geq  e \vrp (b)$ would imply $ e\psi(a) \geq  e
\psi (b)$. Thus $\vrp(a+b) = \vrp(b)$, and  we obtain by D1 that
$\psi(a+b) = \psi(b)$.
\end{proof}

Returning to our given strong valuation $v: R \to M$, let
$\sCov^*(v)$ denote the subset of all $\vrp \in \Cov(v)$ which are
very strong. Lemma \ref{lem11.7} tells us in particular that
$\sCov^*(v)$  is a lower set in the poset $\Cov(v)$, and hence in
$\sCov(v)$. We have
$$\tCov(v) \cap \sCov^*(v)\ = \ \tCov(v) \cap \sCov(v) \ = \ \tsCov(v). $$
\begin{thm}\label{thm11.8}
The tangible strong supervaluation  $\chi : R \to U$  from above
(Proposition \ref{prop11.5}) dominates every very strong
supervaluation covering $v$, and hence is the top element of both
$\sCov^*(v)$ and $\tsCov(v)$.
\end{thm}
\begin{proof} Let  $\psi: R \to V$ be a very strong
supervaluation covering $v$ (in particular $eV = M$). We verify
 axioms  D1-D3 for the pair $\chi$, $\psi$, and then will be
done. D2 is obvious, and  D3 holds trivially since $\chi$ is
tangible. Concerning D1, assume that $\chi(a_1) = \chi(a_2)$. By
definition of $\chi$ this means that $a_1 \vsim a_2$.

We have to prove that $\psi(a_1) = \psi(a_2)$. Either $a_1, a_2
\in \gq$, or there exist $c_1, c_2 \in R$ with $v(c_1) < v(a_1)$,
$v(c_2) < v(a_2)$, $c_1 + a_1 = c_2 + a_2$. In the first case $e
\psi (a_1) = e \psi (a_2) = 0 $ hence  $\psi(a_1) = \psi(a_2) =
0$. In the second case we have
$$\psi(a_1)  = \psi (a_1 + c_1 ) = \psi(a_2 + c_2) = \psi (a_2)$$
since $\psi$ is very strong. Thus $\psi(a_1) = \psi(a_2)$ in both
cases.
\end{proof}

\begin{notation}\label{notation11.9} We denote the semiring $U$
given above by $\bUv$ and the supervaluation $\chi$  given above
by $\bvrp_v$. We call $$ \bvrp_v : R \ \to \ \bUv = \STR(\bRmq, M
\setminus \{ 0 \}, \bv)
$$
 the \textbf{initial very strong supervaluation} covering $v$.
\end{notation}

In this notation
$$ \begin{array}{lll}
   \sCov^*(v) & =  & \{  \psi \in \Cov (v) \ | \ \bvrp_v \geq \psi \},
   \\[2mm]
     \tsCov(v) & =  & \{  \psi \in \Cov (v) \ | \ \bvrp_v \geq \psi \geq \hat v \}.  \\
   \end{array}
$$

Let $E(v)$ denote the equivalence relation on $U(v)$ whose
equivalence classes  are the sets  $[a]_v$  with  $a \in R
\setminus \gq = \tT(U(v))$ and the one point sets $\{ x \}$ with
$x\in M$. In other terms, the restriction $E(v) | \tT(U)$
coincides with $S(v) | \Rmq$, while $E(v) | M$ is the diagonal
$\diag(M)$ of $M$. We identify
$$ U(v) / E(v) = \bUv$$
in the obvious way.

\begin{prop}\label{prop11.10} $E(v) $ is a ghost separating
\MFCE-relation and
$$ \bvrp_v = \vrp_v / E(v).$$
\end{prop}
\begin{proof} It is immediate that  $E(v)$  is \MFCE \ and ghost
separating. For  $a$ in $\Rmq$   we have
$$ \pi_{E(v)} (\vrp_v(a)) = \pi_{S(v)}(a) = [a]_v = \bvrp_v(a)$$
and for $a \in \gq$
$$ \pi_{E(v)} (\vrp_v(a)) = \pi_{E(v)}(a) = 0 = \bvrp_v(a),$$
again, Thus $\pi_{E(v)}$ is the transmission from  $\vrp_v$ to
$\bvrp_v$.
\end{proof}

\begin{cor}\label{cor11.11}
The \MFCE-relations $E$ on $U(v)$   such that $\vrp_v / E$ is very
strong are precisely  all $E \in \MFC(U(v))$ with $E \supset
E(v)$.
\end{cor}
\begin{proof} This is a consequence of our observations above (Lemma \ref{lem11.7},
Theorem \ref{thm11.8}, Proposition \ref{prop11.10}) and the theory
in \S7, cf. Theorem \ref{thm7.2}.
\end{proof}

We now focus on the special case that $R$ is a semifield. Slightly
more generally we assume that every element of $\Rmq$ is
invertible, while $\gq$  may be different from $\{ 0\}$.

$\tT(U(v)) = \Rmq$ is a group under multiplication. Thus the
results from the end of \S8 apply. We have
$$\tT_e(U(v)) = \{ a\in R \ | \ v(a)  = 1_M\} = \go_v^*, $$
with $\go_v^*$ the unit group of the subsemiring
$$\go_v := \{  a \in R \ | \ v(a) \leq 1_M \}$$
of $R$. Notice that the set
$$ \mm_v := \go_v \setminus \go_v^* = \{a \in R \ | \ v(a) < 1_M\}$$
is an ideal of $\go_v$, just as in the classical (and perhaps most
important) case, where $R$ is a field and $v$ is a Krull valuation
on $R$.

By Theorem \ref{thm8.11} and Corollary \ref{cor8.12} we know that
every \MFCE-relation on $U(v)$ except $E(\nu)$ is orbital, hence
ghost separating. We have
$$ \vrp_v /  E(\nu) = v,$$
viewed as a supervaluation. The other supervaluations $\vrp$
covering $v$ correspond uniquely with the subgroups $H$  of
$\go_v^*$ via $\vrp = \vrp_v  /  E(H)$; cf. Corollary
\ref{cor8.12}.

Instead of $U(v) / E(H)$  and $\vrp_v / E(H)$ we now write $U(v) /
H$ and $\vrp_v / H$ respectively.  In this notation
$$ \tT(U(v) / H) = (\Rmq) /  H,$$
and $ \vrp_v / H: R \to \Uv / H$ is given by
$$ (\vrp_v /  H)(a) = \left\{
\begin{array}{lll}
  aH & \text{if} & a \in \Rmq, \\
  0 & \text{if} & a \in \gq.
\end{array}
\right.
$$
\begin{thm}\label{thm11.12} Assume that every element of $\Rmq$ is
invertible  (e.g. $R$ is a semifield).
\begin{enumerate} \eroman
    \item Every strong supervaluation covering $v$ is very strong.
    Except $v$ itself, viewed as a supervaluation, all these
    supervaluations are tangible. In other terms,
    $$ \sCov(v) = \sCov^*(v) = \tsCov(v) \cup \{ v \} .$$

    \item $\bvrp_v = \vrp_v /\langle 1 + \mm_v\rangle$, with $\langle 1 + \mm_v\rangle$
the group generated by $ 1 + \mm_v$ in $\go_v^*$ \footnote{If $R$
is a field, then $\langle 1 + \mm_v\rangle =  1 + \mm_v$.}.
    \item The tangible strong supervaluations $\vrp$  covering $v$
    correspond  uniquely with the subgroup $H$ of
$\go_v^* $ containing the
    semigroup $1 + \mm_v$ via $\vrp = \vrp_v /
    H$. Thus we have  an anti-isomorphism $H \mapsto  \vrp_v / H$
    from the lattice of all subgroups $H$ of $\go_v^*$ containing
    $1 + \mm_v$ to the lattice $\tsCov(v)$.
\end{enumerate}
\end{thm}
\begin{proof} i): Every supervaluation $\vrp$ covering $v$ is
either tangible or $\vrp = v$. Thus, if $\vrp$ is strong, then
$\vrp$ is very strong in both cases.


ii): We know that $\bvrp_v = \vrp_v / E(v)$ (Proposition
\ref{prop11.10}). $E(v)$ is ghost separating,  hence orbital. The
subgroup $H$ of $\go_v^*$ with $E(H) = E(v)$ has the following
description  (cf. Proposition \ref{prop8.8}): If $a \in \Rmq =
\tT(U(v))$, then $ a\in H$ iff $a \vsim 1$. This means that  there
exist elements $c_1, c_2 \in \mm_v$ with $a + c_1 = 1 + c_2$. Now
$a + c_1 = a(1 + d_1)$ with $d_1 = \frac {c_1}{ a} \in \mm_v$.
Thus $a \vsim 1$ iff $a$ is in  the group $\langle 1 + \mm_v
\rangle$.

iii): Now obvious, since $\bvrp_v$ is the top element of
$\tsCov(v)$.

\end{proof}

We look again at the \GS-sentence
  \begin{equation}\renewcommand{\theequation}{$*$}\addtocounter{equation}{-1}\label{eq:str.1}
 \varepsilon_{\vrp(a)} (\tvrp(f)) \ \lmodg \  \vrp(\varepsilon_a(f))
\end{equation}
from \S9, valid for any $\vrp \in \sCov(v)$, $f \in R[\lm]$, $a
\in R^n$, cf. Theorem \ref{thm10.11}. Choosing  here any $\vrp \in
\tsCov(v)$, we learned that $(*)$ implies Kapranov's Lemma
(Corollary \ref{cor10.13}). But the statement  $(*)$  itself has a
different content for  different $\vrp \in \tsCov(v)$. If also
$\psi \in \tsCov(v)$ and $\vrp \geq \psi$, then we obtain
statement $(*)$ for $\psi$ from the statement $(*)$ for $\vrp$,
leaving $f$ and the tuple  $a$ fixed, by applying the transmission
$\al_{\psi, \vrp}$. Thus it seems that $(*)$ has the most content
if we choose for $\vrp$  the initial strong supervaluation
$\bvrp_v: R \to \bUv$.

\medskip
We close this section by an explicit description of $\bUv$ and
$\bvrp_v$ in a situation typically met in tropical geometry. Let
$R : = F\{ t \}$  be the field of formal \textbf{Puiseux series
with real powers} over any field\footnote{For the matter of
geometric applications, one usually needs $F$ to be algebraically
closed, but here we can omit this restriction.} $F$, cf.
\cite[p.6]{IMS}. The elements of $R$ are the formal series
$$ a(t) = \sum_{j \in I } c_j t^j$$
with $c_j \in F^*$ and $I \subset \R$ a well ordered set, in set
theoretic sense, (including $I =  \emptyset$). Let further $M$ be
the bipotent semifield  $T(\R_{>0})$ (cf. Theorem \ref{thm1.4}),
i.e.,
$$ M = \R_{>0} \cup \{ 0  \}  = \R_{\geq 0},$$
with the max-plus structure.

We define a (automatically strong) valuation $$ v : F \{ t \}  \to
M$$ by putting
$$ v(a(t)) :=
\vth^{\min(I)}
$$
if $a(t) \neq 0$, written as above, and $v(0) := 0$. Here $\vth$
is a fixed real number with $0 < \vth < 1$ (cf. \cite{IMS}) loc.
cit, but we use a multiplicative notation).  Now $\go_v^*$ is the
group consisting of all  series
$$ a(t) = c_0 + \sum_{j > 0}c_jt^j, \qquad c_0 \neq 0,$$
in $F\{ t \} $, and $1 + \mm_v$ is the subgroup consisting of
these series with $c_0 = 1$.

The equivalence relation $S(v)$ on $R^* = \tT(U(v))$  is given by
$$ a(t) \vsim b(t) \ \Longleftrightarrow \ \frac{ a(t)} { b(t)} \in 1 + \mm_v . $$
This means that the series $a(t)$ and $b(t)$  have the same
leading term $\ell(a(t)) = \ell (b(t))$. Thus the group of
monomials
$$ G :=  \{  ct^j \ |  \  c \in F^*, \ j \in \R \}$$
is a system of representatives of the equivalence classes of
$S(v)$. We identify
$$ G = R^* /  S(v) = \tT(U(v)) / E(v) .$$
Then $\bUv = \STR(G, \R_{>0 }, v|G) = G \dot\cup M$ in the
notation of Construction \ref{constr3.14}, and our supervaluation
$\bvrp_v : R \to \bUv$  is the map  $a(t) \mapsto \ell(a(t))$,
which sends each formal series $a(t)$ to its leading term. \{We
read $\ell(0) = 0$, of course.\}

In short, applying $v$ to a series $a(t)$ means taking  its
leading $t$-power  and replacing $t$ by $\vth$, while applying
$\bvrp_v$ means taking its leading term.

Similarly we can interpret  the bottom supervaluation $\hat v \in
\tsCov(v)$. The $t$-powers $t^j$, $j \in \R$, are a multiplicative
set of representatives of the $\tE$-equivalence classes.
Identifying
$$ \Uv / \tE = \{ t^j \ | \ j \in \R \}, $$
we can say that $\hat v(a(t))$ is the leading $t$-power of the
series $a(t)$. The ghost map from $\Uv/ \tE = D(M)$ to $M$ sends
$t$ to $\vth$.

\section{Iq-valuations on polynomial semirings and related supervaluations.}\label{sec:13}

Since the semiring of polynomials over a supertropical domain is
no longer supertropical (or analogously, the semiring of
polynomials over a bipotent semiring is no longer bipotent), we
would like a theory generalizing valuations to maps with values in
these  polynomial semirings. Unfortunately, the target is no
longer an ordered group (and is not even an ordered monoid). In
this section, we formulate some concepts of this paper
 in the more general context of monoids with a supremum, instead
 of ordered monoids, and show  how this encompasses
 Kapranov's
 Lemma.

 Recall that an operation $a\vee b$ on a set $S$ is  called
 a \textbf{sup} if it has a distinguished element~$0$
and satisfies the following properties for all $a,b,c\in S$:
 \begin{enumerate} \item  $0 \vee a = a;$
\item $a\vee b = b \vee a;$  \item $a \vee a = a;$ \item $a \vee
(b \vee c) = (a\vee b) \vee c.$
\end{enumerate}

In this case, we can define a partial order on $S$ by defining $a
\le b$ when $a \vee b = b.$ Then the following properties are
immediate for all $a,b,c \in S$:
 \begin{enumerate}\ealph
\item $0 \le a$; \item $a \vee b \ge a$ and $a \vee b \ge b;$
\item if $a \le c$ and $b \le c$, then $a \vee b \le c.$ (Indeed,
if $a\vee c = c$ and $b\vee c = c,$ then $(a\vee b) \vee c =
(a\vee c)\vee (b\vee c) = c\vee c = c.$)\end{enumerate}

 We also say that a given sup  $x\vee y$ on a monoid
$M$ is \textbf{compatible} with $M$ if $a(x\vee y) = ax \vee ay$
for all $a,x,y \in M$.

In order to axiomatize this in the language of semirings, we
recall that an \textbf{idempotent semiring} $R$ satisfies the
property that $x+x = x$ for all $x\in R$.

\begin{prop}\label{sup1} $ $
\begin{enumerate} \eroman
    \item  Every idempotent  semiring $R$ can be viewed as a
multiplicative monoid with a compatible sup   $\vee$ defined by
$$x \vee y : = x+y.$$

    \item Conversely, given   a monoid $M$ with a compatible sup, we can define an idempotent semiring structure on
$M$, with the same multiplication, and with addition given by $x+y
:= x \vee y$. \
\end{enumerate}
 \end{prop}
\begin{proof} All of the
other verifications are immediate.
\end{proof}

\begin{rem} If $R$ is an idempotent semiring, then so is the
polynomial semiring $R[\lm ]$ as well as the matrix semiring
$M_n(R)$.
\end{rem}

Both of these assertions fail when we substitute ``bipotent'' for
``idempotent.'' Thus, it makes sense to pass to idempotent
semirings when studying
polynomials and matrices. In the case of semifields, we actually have a lattice
structure.

\begin{prop} If $R$ is a semifield, where $\vee$ is given by addition (as in Proposition \ref{sup1}), then there is a compatible inf
relation $\wedge$ given by $x \wedge y := \frac {xy}{x+y}$ (taking
$0 \wedge 0 = 0$), thereby making $(R,\vee,\wedge)$ a distributive
lattice satisfying \begin{equation}\label{mult} (x\vee y)(x\wedge
y) = xy, \quad \forall x,y\in R.\end{equation}\end{prop}
\begin{proof} Property \eqref{mult} follows at once from the definitions,
and implies that $a(x\wedge y) = ax \wedge ay$, as well as
associativity of $\wedge.$ To check distributivity, we need to
check
$$(x\wedge y )\vee z = (x \vee z) \wedge ( y \vee z).$$
Since $\le$ is clear, we only check $\ge,$ and also may assume
$x,y,z \ne 0.$ Now \begin{equation}\begin{aligned}(x\wedge y )\vee
z & =
\frac{xy}{x+y} + z  \\ & \ge \frac{xy}{x+y+z} + z\frac {x+y+z}{x+y+z} \\
& = \frac{(x+z)(y+z)}{x+y+z} = \frac{(x+z)(y+z)}{(x+z)+(y+z)}
\\ & =(x \vee z) \wedge ( y \vee z).\end{aligned}\end{equation}
\end{proof}

Having the translation of the sup relation to semirings at hand,
we are ready to reformulate some of the results of this paper. But
first it is instructive to introduce a parallel of the \UTGR.

\begin{defn}\label{defnep.11}
$y \lmod x \ \Leftrightarrow \ \exists a \in R \ \text{with} \ y =
 x+a.$
\end{defn}

Clearly, $\lmod$ is a transitive binary  relation on $R$.
\begin{defn}\label{def:13.5} $R$ is an \textbf{upper-bound} semiring, written
ub-semiring, if the relation $\lmod$ is anti-symmetric; i.e.,
$$x \lmod y \ \text{and} \  y \lmod x \quad \Leftrightarrow \quad  x=y.$$
\end{defn}
 The reason for this terminology is that now the relation $\lmod$
 gives a partial ordering  on the set $R$ $$a \leq b  \ \  \text{iff}  \
 \ b \lmod a
  \ \  \text{iff}  \ \  \exists c \in R: \ a+c = b,$$
 and $x + y$ is an upper bound of $x,y$ in this ordering\footnote{All inequalities in the following will refer to this ordering.}.

\begin{rem}\label{rmk:13.6} $ $
\begin{enumerate} \eroman
    \item The condition that a semiring $R$ is ub can be rephrased as
follows:

For any $a,b,x \in R,$ if $x+a+b =x,$ then $x+a = x.$

    \item Any ub-semiring $R$ has the property that $a+b= 0$ implies $a=b=0$, by (i). (Take $x=0$.)
\end{enumerate}
\end{rem}

\begin{prop} Any idempotent semiring is an ub-semiring.
\end{prop}
\begin{proof} If $x+a+b =x,$ then $$x+a = (x+ a +b )+a = x+ a +b = x.$$
\end{proof}

%


If $R$ is any semiring, let $R[\lm] = R[\lm_1, \dots, \lm_n]$
denote the polynomial semiring over $R$ in a set of variables $\lm
= (\lm_1, \dots, \lm_n)$.

\begin{prop}\label{prop:13.11}
Every supertropical semiring $U$ is upper bound, and $U[\lm_1,
\dots, \lm_n]$ is upper bound for every $n$.

\end{prop}
\begin{proof}

We have to check the condition in Remark \ref{rmk:13.6}.i. Let
$x,a,b \in U$ be given with $x + a+ b = x$. We have to verify that
$x + a = x$. Multiplying by $e$ we obtain $ex + ea+ eb = ex$,
hence $ea \leq ex$ and $eb \leq ex$. If $ea < ex$, then  $x+a = x$
right away.If $eb < ex$, then  $x + b = x$, hence $x =x + a + b =
x +a$ again.  There remains the case that $ea=eb=ex$. Now $x +a +b
= ex$, hence $x$ is ghost, and $x+a = ex =x$ again. This proves
that $U$ is ub.

Let now $f,g,h \in U[\lm_1, \dots, \lm_n]$ be given with $f+g+h =
f$. We write $f = \sum \al_i \lm^i$, $g = \sum \bt_i \lm^i$, $h =
\sum \gm_i \lm^i$. Then $\al_i + \bt_i + \gm_i =\al_i$ for every
$i$, and we conclude that $\al_i + \bt_i = \al_i$ for every $i$,
hence $f+g =f$, as desired.
 \end{proof}

 The reason we want to consider the
idempotent semiring $M[\lm]$ is that we want to extend any
m-valuation $v: R \to M$ to the map $\tv : R[\lm] \to M[\lm]$,
where we define
\begin{equation}\label{tild} \tv  \bigg (\sum_i \a_i \lm _1 ^{i_1 }\dots
\lm _n ^{i_n } \bigg) = \sum_i v(\a_ i) \lm _1 ^{i_1 }\dots \lm _n
^{i_n } .\end{equation} Since $M[\lm]$ is no longer bipotent in
the natural way, we would like to generalize Definition
~\ref{defn2.1} to permit valuations to  idempotent semirings.

Unfortunately, $\tv $ as defined in \eqref{tild} need not satisfy
property V3 of Definition~\ref{defn2.1}, since $\tv  (fg)$ could
differ from $\tv  (f )\tv  (g).$
%
Indeed, if $f = \sum_i \a_i \lm^{i}$ and $g = \sum_j \bt_j \lm^{j
},$ with $i = (i_1, \dots,i_n)$ and $j = (j_1, \dots, j_n)$, then
writing $fg = \sum_k \left(\sum _{i+j =k} \a_i\bt_{j}\right) \lm
^{k },$ we have
\begin{equation*}\begin{aligned} \tv (fg) & =
\sum_k v\!\bigg(\sum _{i+j = k } \a_i\bt _{j}\bigg) \lm ^{k }
\\ & \le  \sum _k \sum _{i+j =k}  v(\a_i)v(\bt
_{j}) \lm ^{k }  \\ & = \bigg(\sum v(\a_ i)\lm^{i}
\bigg)\bigg(\sum v(\bt_ j)\lm^{j}
\bigg),\end{aligned}\end{equation*} where there could be strict
inequality. (Notice that our partial oredering on $M[\lm]$ extends
the total ordering of $M$.) Accordingly, we need a weaker notion:

\begin{defn}\label{def:13.10} An \bfem{iq-valuation} (= idempotent monoid
quasi-valuation) on a semiring~$ R$ is a map $v: R\to M$ into a
(commutative) idempotent semiring $M\ne\{0\}$ with the following
properties:
\begin{align*}
&IQV1: v(0)=0,\\
&IQV2: v(1)=1,\\
&IQV3:v(xy)\le v(x)v(y)\quad\forall x,y\in R,\\
&IQV4: v(x+y)\le v(x)+v(y)\quad \forall x,y\in R.
\end{align*}
\end{defn}
\{NB: Here as elsewhere we use the partial order introduced above
following Definition~\ref{def:13.5}.\}\vskip 1mm

The following is now obvious.

\begin{prop}\label{iqval} Suppose $M$ is a bipotent semiring and  $v: R \to M$
is an  \m-valuation.

\begin{enumerate} \eroman
    \item Then the map $\tv :
R[\lm ] \to M[\lm ]$  given above is an  iq-valuation.

    \item For any given $a \in M^{n}$, the map $\ep_a \circ \tv  : R[\lm] \to M$ is
    again an iq-valuation. \{Here  $\ep_a$ denotes the evaluation map $f(\lm) \mapsto f(a)$,
    as in the previous  sections.\}
\end{enumerate} \endbox
\end{prop}
If $v$ is strong we can do better.

\begin{thm}\label{thm13.12}
 Assume that  $v: R \to M$ is a surjective strong \m-valuation.
Then, for any $a \in M^n$, $\ep_a \circ \tv  : R[\lm] \to M$ is
again a strong \m-valuation.
\end{thm}

\begin{proof}
By an easy induction we restrict to the case of $n=1$. Given $f =
\sum _i \al_i \lm^i $, $g = \sum _j \bt_i \lm^i $ in $R[\lm]$ we
have to verify the  following:
\begin{enumerate}
    \item $\ep_a\tv  (fg ) = \ep_a  \tv (f)
    \cdot \ep_a  \tv (g)$;
    \item If $\ep_a  \tv (f) <  \ep_a  \tv (g)$,
    then $\ep_a  \tv (f+g)   = \ep_a  \tv (g)$.
\end{enumerate}

\smallskip
 (1): We know  already  by Proposition \ref{prop11.10}  that
$$ \ep_a\tv  (fg ) \leq \ep_a  \tv (f)
    \cdot \ep_a  \tv (g).$$
Due to the bipotence of $M$ we have smallest indices $k$ and
$\ell$ such that
$$
\begin{aligned} \ep_a\tv  (f ) =  \sum_i v(\al_i)a^i =
v(\al_k)a^k, \\
\ep_a\tv  (g ) =  \sum_j v(\bt_i)a^j = v(\bt_\ell)a^\ell.
\end{aligned}
$$
We chose some $c \in R $ with $v(c) = a$. Since $v$ is strong and
$k,\ell$ have been chosen minimally we have

$$ v \bigg( \sum_{i+j =  k + \ell} \al_i c^i \bt_j c^j\bigg) = v(\al_k c^k \bt_\ell c^\ell) = \ep_a  \tv (f)
    \cdot \ep_a  \tv (g).$$
    Thus
    $$\begin{aligned}
\ep_a  \tv (fg) & =  \sum_r v \bigg(
    \sum_{i + j =r} \al_i \bt_j
    \bigg) v(c)^r \\
     & =  \sum_r v \bigg(
    \sum_{i + j =r} \al_i c^i \bt_j c^j
    \bigg) \\ & \geq  \sum_{i+j = k + \ell} v \bigg(
    \sum_{i + j =k} \al_i c^i \bt_j c^j
    \bigg) \\[3 mm]
    & =  \ep_a  \tv (f)
    \cdot \ep_a  \tv (g).
    \end{aligned}$$
We conclude that
    $$  \ep_a  \tv (fg)
    =  \ep_a  \tv (f)
    \cdot \ep_a  \tv (g).$$

\smallskip
  $(2)$: Assume that  $ \ep_a  \tv (f) <  \ep_a  \tv (g)
 $. Using the same $k,\ell$, and $c$ as before we have for all $i$
$$
\begin{aligned} v(\al_ic^i) < v(\bt_\ell c^\ell), \\
 v(\bt_i c^i) \leq v(\bt_\ell c^\ell). \end{aligned}
$$
Now
$$ \ep_a  \tv (f+g)  = \sum_i v\((\al_i + \bt_i)c^i\),$$
and $v\((\al_i + \bt_i)c^i\) \leq v(\bt_\ell c^\ell)$ for all $i$,
but
$$v\((\al_\ell + \bt_\ell)c^\ell\) =  v(\bt_\ell c^\ell).$$
Thus,
$$ \ep_a  \tv (f+g) =  v(\bt_\ell c^\ell )= \ep_a  \tilde
v(g).$$
\end{proof}

 In particular, we could take
$v$ to be the natural  valuation on the field of  Puiseux series
with rational exponents, as used in \cite{Gathmann:0601322}, or
with real exponents as introduced above in \S11.

Let us formulate the analogue of Definition~\ref{defn4.1} in the
realm of semirings with ghosts.

\begin{defn}\label{edefn4.1}  An  iq-\bfem{supervaluation} on a semiring $R$ is a map $\varphi: R\to U$
  from $R$ to a  ub-semiring $U$ with ghosts, satisfying the following
  properties.
 \begin{alignat*}{2}
&IQSV1:\ &&\varphi(0)=0,\\
&IQSV2:\ &&\varphi(1)=1,\\
&IQSV3:\ &&\forall a,b\in R: \varphi(ab)\le \varphi(a)\varphi(b),\\
&IQSV4:\ &&\forall a,b\in R: e\varphi(a+b)\le
e(\varphi(a)+\varphi(b)).\end{alignat*}
\end{defn}

Here again we use the ordering given by the relation  $ \lmodg $.
The definition works in particular for $U$ a supertropical
semiring and to  Proposition~\ref{prop:13.11}.

We are ready for the main purpose of this section.
\begin{thm} Assume that $\varphi: R \to U$
is a surjective strong supervaluation, and $$v: R \to eU =M$$ is
the strong \m-valuation covered by $\vrp$. Let $a =
(a_1,\dots,a_n) \in U^n$ be given, and let $b:= (ea_1,\dots,ea_n)
\in M^n$.
\begin{enumerate}\eroman
    \item $\vrp$ can be extended to an iq-supervaluation $\tvrp : R[\lm] \to
    U[\lm]$ by the  formula
$$\tvrp  \bigg (\sum_i \a_i \lm^{i} \bigg) = \sum_i
\varphi(\a_i)\lm ^{i} .$$

    \item $\ep_a  \circ \tvrp  : R[\lm] \to U$ is  a strong
    supervaluation. It covers the (strong) valuation $\ep_b \circ
    \tv  : R[\lm] \to M$.
\end{enumerate}
\end{thm}

\begin{proof} (i): If $a,b \in R $ then we know from \S9 that $\vrp(a) + \vrp(b) \lmodg \vrp(a+b)$.
This implies $\vrp(a) + \vrp(b) \lmod \vrp(a+b)$, i.e.
 \begin{equation}\renewcommand{\theequation}{$*$}\addtocounter{equation}{-1}\label{eq:str.1}
\vrp(a+b) \leq \vrp(a) + \vrp(b).
\end{equation}

An argument parallel to the one before Definition \ref{def:13.10}
now tells us that for $f,g \in R[\lm]$ we have
$$\tvrp(fg) \leq \tvrp(f) \cdot \tvrp(g).$$
Clearly $\tvrp$ extends $\vrp$, in particular $\tvrp(0) = 0$,
$\tvrp(1) =1 $. From $(*)$ it is also obvious that $\tvrp(f+g)
\leq \tvrp(f) + \tvrp(g)$, hence $$e\tvrp(f+g) \leq e\tvrp(f) +
e\tvrp(g).$$ Thus, $\tvrp$ is an iq-supervaluation. Clearly
$e\tvrp(f) = \tv(f)$ for all $f \in R[\lm]$. \{By the way, this
gives us again that $e\tvrp(f + g) \leq e\tvrp(f) + e\tvrp(g)$.\}
\smallskip

 (ii): Again  we restrict to the case of $n=1$ by an easy
induction. It is pretty obvious that $\ep_a  \tvrp: R[\lm] \to U$
obeys the rules SV1, SV2, SV4 from \S4 (Definition \ref{defn4.1}),
and $e \cdot \ep_a  \tvrp(f) = \ep_b  \tv(f)$ for every $f \in
R[\lm]$. Given $f = \sum_i \al_i \lm ^i $, $g = \sum_i \bt_i \lm
^i $ in $R[\lm]$ it remains to prove the following:
\begin{enumerate}
    \item $\ep_a  \tvrp(fg) = \ep_a  \tvrp(f) \cdot \ep_a
    \tvrp(g)$,
    \item If $\ep_a  \tvrp(f) \leq \ep_a  \tvrp(g)$ then $\ep_a  \tvrp(f+g) = \ep_a
    \tvrp(g)$.
\end{enumerate}

 (1): Let $k,\ell$  be the minimal indices such that
 \begin{equation}\renewcommand{\theequation}{$**$}\addtocounter{equation}{-1}\label{eq:str.1}
e \sum_i \vrp(\al_i) a^i = e \vrp(\al_k)a^k = e \ep_a \tvrp(f),
\end{equation}
 \begin{equation}\renewcommand{\theequation}{$***$}\addtocounter{equation}{-1}\label{eq:str.1}
e \sum_i \vrp(\bt_i) a^i = e \vrp(\bt_\ell)a^\ell = e \ep_a
\tvrp(g),
\end{equation}
(as in the proof of Theorem \ref{thm13.12}). We know by Theorem
\ref{thm13.12} that
$$ e ( \ep_a \circ \tvrp)(fg) = e \vrp(\al_k) a^k \cdot e \vrp(\bt_\ell) a^\ell
= e ( \ep_a \circ \tvrp)(f) \cdot e ( \ep_a \circ \tvrp)(g).$$

We chose some $c \in R$  with $\vrp(c) = a$. Using $(*)$ we obtain
$$
 \begin{aligned} ( \ep_a \circ \tvrp)(fg) & = \sum_r \vrp\bigg( \sum_{i+j =r} \al_i
 \bt_j\bigg) a^r \\
 & = \sum_r \vrp\bigg( \sum_{i+j =r} \al_i c^i \cdot
 \bt_j c^j \bigg) \\
  & \leq  \sum_r  \sum_{i+j =r} \vrp(\al_i c^i) \cdot
 \vrp(\bt_j c^j) \\
 & =    \sum_{i,j} \vrp(\al_i) a^i \cdot
 \vrp(\bt_j) a^j.
 \end{aligned}
$$
There is a single $\nu$-dominating term in this sum iff there is a
single $ \nu$-dominating term on the left of $(**)$ and of
$(***)$, so we conclude that $$\ep_a  \tvrp(fg) = \ep_a  \tvrp(f)
\cdot \ep_a
    \tvrp(g)$$
in all cases, using the fact that tangible elements $x,y$ of $U$
with $x \leq y$, $ex= ey$ are equal.

\smallskip  (2): This can be proved in the way analogous
to claim (2) in the proof of Theorem~\ref{thm13.12}.
\end{proof}

Thus, for $U$ a supertropical semiring, the evaluation map returns
us from  iq-supervaluations with values in $U[\lm]$ to the firmer
ground of supervaluations.

%
%
%

Looking again at Theorem \ref{thm10.11} we realize now that the
theorem gives pleasant examples of pairs of supervaluations which
obey a ``GS-relation" in the following sense.

\begin{defn}\label{def:13.4}
If $\rho:A \to V$ and $\sig:A \to V$  are supervaluations on a
semiring $A$ with values in the same supertropical semiring $V$,
then we say that \textbf{$\rho$ surpasses $\sig$  by ghost}, and
write $\rho \lmodg \sig$, if $\rho(a) \lmodg \sig(a)$ for every $a
\in A$.
\end{defn}

In this terminology Theorem \ref{thm10.11} reads as follows:
\begin{thm}\label{thm:13.5} Let $\vrp: R \to U$ be a strong
supervaluation. Then for any $a \in R^n$ the supervaluation $
\varepsilon_{\vrp(a)} \circ \tilde \vrp: R[\lm_1,\dots,\lm_n] \to
U$ surpasses the  supervaluation $ \vrp \circ \varepsilon_a :
R[\lm_1,\dots,\lm_n] \to U$ by ghost.
\end{thm}

Of course, we should look for other examples of pairs of
supervaluations $\rho:A \to V$ and $\sig:A \to V$ with $\rho
\lmodg \sig$. Here the ``classical" case that $A$ is a semifield,
or even a field, and $eV$ is cancellative, is perhaps not the most
interesting  one. Indeed, for such pairs $\rho$, $\sig$ we have $e
\rho(a) \geq e \sig(a)$ for every $a \in A$, and this forces $e
\rho(a) = e \sig(a)$ since for $a \neq 0$ also $e \rho(a^{-1})
\geq e \sig(a^{-1})$. Thus $\rho$ and $\sig$ cover the same
valuation  $e \rho = e \sig: A \to eV$. But for the pairs
occurring in Theorem \ref{thm:13.5}, where $A$ is a polynomial
semiring, the valuation $e\rho$ and $e \sig$ usually will have
even  different support, and $\rho$ can be a very interesting
``perturbation" of $\sig$ by ghosts.

The phenomenon of ``surpassing by ghost'' for supervaluations
shows clearly the importance of studying valuations and
supervaluations on semirings instead of just semifields.

\end{document}